\documentclass[11pt]{article}
\usepackage{graphicx} 

\title{The canonical structures of the limit of the Yang-Mills flows for nef and big classes}
\author{Satoshi Jinnouchi \thanks{Department of Mathematics, Graduate School of Science, The University of Osaka,
1-1, Machikaneyama-cho, Toyonaka, Osaka 560-0043, Japan.
email:{{\tt u122988d[@]ecs.osaka-u.ac.jp}},
email:{{\tt 20160312sti[@]gmail.com}}}}
\date{August 2025}

\usepackage{amssymb,amsmath,amsthm}    
\usepackage[abbrev]{amsrefs} 
\usepackage{mathrsfs}
\usepackage[bbgreekl]{mathbbol}
\usepackage{comment}
\usepackage{color}
\usepackage{tikz}
\usepackage{mathtools}
\usepackage{appendix}
\usepackage[all]{xy} 
\usetikzlibrary{positioning}
\usetikzlibrary {arrows.meta}

\newtheorem{theo}{Theorem}[section]
\newtheorem{lemm}[theo]{Lemma}
\newtheorem{corr}[theo]{Corollary}
\newtheorem{prop}[theo]{Proposition}

\numberwithin{equation}{section}

\theoremstyle{definition}
\newtheorem{defi}[theo]{Definition}

\newtheorem{rema}[theo]{Remark}

\allowdisplaybreaks

\newcommand{\rk}{{\rm{rk}}}

\newcommand{\Amp}{{\rm{Amp}}}

\newcommand{\codim}{{\rm{codim}}}

\newcommand{\id}{{\rm{id}}}

\newcommand{\loc}{{\rm{loc}}}

\newcommand{\exc}{{\rm{exc}}}

\newcommand{\End}{{\rm{End}}}

\newcommand{\Tr}{{\rm{Tr}}}

\newcommand{\wtil}{\widetilde}

\newcommand{\im}{{\rm im}}

\newcommand{\delbar}{\bar{\partial}}
\newcommand{\del}{\partial}
\newcommand{\ve}{\varepsilon}
\newcommand{\an}{{\rm an}}
\newcommand{\Gr}{{\rm Gr}}
\newcommand{\HNS}{{\rm HNS}}

\newcommand{\YM}{{\rm YM}}
\newcommand{\HYM}{{\rm HYM}}

\usepackage[colorlinks=true, linkcolor=blue, citecolor=blue]{hyperref}

\usepackage{cite}

\begin{document}
\date{\empty}
\maketitle
\begin{abstract}
In the previous paper \cite{Jin26}, the author introduced the notions of an adapted current $T$ and an adapted Hermitian-Einstein metric to establish the Kobayashi-Hitchin correspondence for a nef and big class $\alpha$. 
As a continuation of the previous work, this paper studies the solvability and the convergence of the Yang-Mills flow for a nef and big class $\alpha$ on a holomorphic vector bundle $E$ over a compact K\"{a}hler manifold $X$. In particular, we show that the limit of the Yang-Mills flow at infinity is determined by the holomorphic structure of $E$ and the nef and big class $\alpha$. More precisely, if we fix an integrable unitary connection $A_0$ on $E$, we show that the $T$-Yang-Mills flow on $E$ with initial condition $A_0$ is solvable for all time and it converges to a $T$-Yang-Mills connection $A_{\infty}$ in the sense of Uhlenbeck limit. Furthermore, we also show that, on the ample locus of $\alpha$, $A_{\infty}$ is complex-gauge equivalent to the direct sum of the Chern connections of the $T$-adapted Hermitian-Einstein metrics on the factors of the graded sheaf associated with the $\alpha^{n-1}$-Harder-Narasimhan-Seshadri filtration of $E$. 
\end{abstract}
\tableofcontents
\setcounter{theo}{0}
\renewcommand{\thetheo}{\Alph{theo}}
\section{Introduction}
This paper establishes the solvability and the convergence of the Yang-Mills flow for a nef and big class. In the K\"{a}hler setting, Daskalopoulos-Wentworth \cite{DW04} and Sibley \cite{Sib15} proved that the Yang-Mills flow converges in the sense of Uhlenbeck limit, and the limit depends only on the holomorphic structure of the original bundle, as conjectured by Bando-Siu \cite{BS94}. 
This paper generalizes their results to the nef and big class setting. A nef and big class appears as a limit of a sequence of K\"{a}hler classes. In complex and algebraic geometry, a base point free nef and big line bundle plays an important role. For example, the canonical divisors of minimal models of general type are base point free nef and big. 
However, there exist several examples of non-base point free nef and big line bundles and there is a natural setting in which the nef and big canonical divisors are not base point free (see e.g. \cite[Example 5.5]{ACSS22}). Therefore, it is natural to study global geometric properties via general nef and big line bundles (or more generally nef and big classes).
When working with a nef and big class, a major difficulty is that a metric representing the class is not necessarily positive definite and moreover has singularities whose model metrics are not known. To overcome this difficulty, we use the notion of an adapted current introduced in \cite{Jin26}, which admits a suitable approximation by smooth K\"{a}hler metrics.
This paper is a continuation of the author\rq{}s previous work \cite{Jin26} which established the Kobayashi-Hitchin correspondence for nef and big classes.

Let us quickly review the study of the Yang-Mills connections and the relationship with the Kobayashi-Hitchin correspondence (see also the introduction of \cite{Sib15}).
A Yang-Mills connection is one of the differential geometric canonical objects on a smooth hermitian vector bundle. The Yang-Mills connection arises as a critical point of the Yang-Mills energy which is defined as an integral of the squared norm of the curvature of an integrable unitary connection on the hermitian vector bundle. Therefore, it is natural to expect that the Yang-Mills connection can be obtained as a limit of the Yang-Mills flow which is the gradient flow of the Yang-Mills energy. 
This was proved by Daskalopoulos \cite{Das92} in the case of a Riemann surface.

When $(X,\omega)$ is a higher dimensional K\"{a}hler manifold, the convergence of the Yang-Mills flow fails in the usual sense.
Donaldson reformulated the Yang-Mills flow as a parabolic PDE of hermitian metrics on a holomorphic vector bundle, called the Hermitian-Yang-Mills flow, and proved its long time existence. Furthermore, on a projective manifold, he also established the smooth convergence of the Hermitian-Yang-Mills flow to a Hermitian-Einstein metric in the case where the holomorphic bundle is stable \cite{Don85}, \cite{Don87}. S. Kobayashi and M. L{$\rm\ddot{u}$}bke proved that Hermitian-Einstein bundles are stable (refer to \cite[section 5.8]{Kob87}).
In this way, the study of the Yang-Mills connections is closely related to the Kobayashi-Hitchin correspondence, namely, the equivalence between the existence of Hermitian-Einstein metrics and stability of holomorphic vector bundles. 

The Kobayashi-Hitchin correspondence of Donaldson was generalized to compact K\"{a}hler manifolds by Uhlenbeck-Yau \cite{UY86} using a different method. Finally Bando-Siu \cite{BS94} extended the correspondence to reflexive sheaves on compact K\"{a}hler manifolds. In \cite{BS94}, they also conjectured that the Yang-Mills flow will converge to a coherent sheaf defined by the Harder-Narasimhan filtration of the original holomorphic bundle. Here the Harder-Narasimhan filtration is the filtration of subsheaves whose successive quotients are semistable (see Definition \ref{HNS def2} for precise definition). The study of the relationship between the Yang-Mills flows and the Harder-Narasimhan filtrations dates back to Atiyah-Bott \cite{AB86}. In their study, the stratification of the space of holomorphic structures on the hermitian vector bundle by the Harder-Narasimhan types is important  (refer to \cite{Sib15}), where the Harder-Narasimhan type is the tuple of real numbers $(\mu_1,\ldots,\mu_l)$ which is defined by the Harder-Narasimhan filtration (refer to Definition \ref{HNS def}). 

The conjecture of Bando-Siu was solved by Daskalopoulos-Wentworth \cite{DW04} for K\"{a}hler surfaces and by Sibley \cite{Sib15} for compact K\"{a}hler manifolds. Their approach is to use Hermitian-Yang-Mills type functionals to distinguish Harder-Narasimhan types, an idea that first appeared in \cite{AB86}. In this paper, we refine their ideas and generalize their results to the nef and big class setting.

We now introduce the main results of this paper. Since the cohomology class $\alpha$ which we consider is nef and big, a metric representing $\alpha$ need not be smooth and strictly positive. It is difficult to study PDEs with respect to such metrics in general. We therefore choose an adapted current, which is introduced in \cite{Jin26}, as a metric representing a nef and big class. An adapted current is, roughly speaking, a closed positive $(1,1)$-current which is a smooth K\"{a}hler metric on the complement of the non-K\"{a}hler locus $E_{nK}(\alpha)$, degenerates along $E_{nK}(\alpha)$ with at most polynomial order, and admits a suitable approximation by K\"{a}hler metrics called an adapted approximation (refer to Definition \ref{adapted defi} for a more precise definition). First, we prove the solvability of the Yang-Mills flow with respect to an adapted current. On a smooth vector bundle $E$ with a holomorphic structure $\delbar_E$, we denote the Chern connection of a hermitian metric $h$ by $(\delbar_E,h)$.
\begin{theo}\label{main thm2}
Let $X$ be a compact K\"{a}hler manifold, $\alpha$ be a nef and big class on $X$ and $T$ be an adapted current in $\alpha$. Let $(E,\delbar_E)$ be a holomorphic vector bundle on $X$ with a smooth hermitian metric $h_0$. 
Then, the $T$-Yang-Mills flow with initial condition $A_0=(\delbar_E, h_0)$ admits a solution $A_t$ which is smooth on $(X\setminus E_{nK}(\alpha)) \times [0,\infty)$:
$$
\frac{\partial A_t}{\partial t}=-d_{A_t}^{*_{h_0,T}}F_{A_t}, \hspace{4mm} A_{t=0}=A_0.
$$
Here $d_{A_t}^{*_{h_0,T}}$ denotes the formal adjoint of the covariant derivative $d_{A_t}$ with respect to the $L^2$-inner product defined by $h_0$ and $T$.
\end{theo}
A holomorphic vector bundle $(E,\delbar_E)$ admits an $\alpha^{n-1}$-Harder-Narasimhan-Seshadri filtration which is a refinement of the $\alpha^{n-1}$-Harder-Narasimhan filtration so that the successive quotients are $\alpha^{n-1}$-slope stable (refer to Definition \ref{HNS def2} for a more precise definition).  We denote by $\Gr^{\HNS}_{\alpha}(E, \delbar_E) =\bigoplus_iQ_i$ the associated graded sheaf. By definition, each factor $Q_i$ is $\alpha^{n-1}$-slope stable. Then, by the Kobayashi-Hitchin correspondence for nef and big classes, each $Q_i$ admits a $T$-adapted Hermitian-Einstein metric $h_i$ (refer to Definition \ref{adapted HYM defi} and Theorem \ref{KH corr nef big}). The goal of this paper is to prove that the $T$-Yang-Mills flow in Theorem \ref{main thm2} converges to a $T$-Yang-Mills connection $A_{\infty}$ in the sense of Uhlenbeck limit, and $A_{\infty}$ is complex-gauge equivalent to the direct sum of Chern connections $(\delbar_{Q_i},h_i)$ of the $T$-adapted Hermitian-Einstein metrics $h_i$ on $Q_i\subset\Gr^{\HNS}_{\alpha}(E,\delbar_E)$.
\begin{theo}\label{main thm}
Let $X$ be a compact K\"{a}hler manifold and $\alpha$ be a nef and big class on $X$. Let $T$ be an adapted current in $\alpha$. Let $(E,\delbar_E)$ be a holomorphic vector bundle on $X$ with a smooth hermitian metric $h_0$. Fix $p>n:=\dim X$. Let $(A_t)_{t\in[0,\infty)}$ be the $T$-Yang-Mills flow on $(E,h_0)$ with initial condition $A_0=(\delbar_E,h_0)$ in Theorem \ref{main thm2}. 
Then there exists a closed subset $Z_{\an} \subset X\setminus E_{nK}(\alpha)$ of Hausdorff codimension at least 4 and a smooth hermitian vector bundle $(E_{\infty},h_{\infty})\to X\setminus (E_{nK}(\alpha)\cup Z_{\an})$ with a $T$-Yang-Mills connection $A_{\infty}$ satisfying the following properties:
\begin{itemize}
\item there exists a sequence $t_j\to \infty$ and a sequence of isometries $u_j: E|_{X\setminus (E_{nK}(\alpha)\cup Z_{\an})}\to E_{\infty}$  such that $u_j\cdot A_{t_j}\xrightarrow{j\to\infty} A_{\infty}$ weakly in $L^p_{1,\loc}(X\setminus (E_{nK}(\alpha)\cup Z_{\an}))$.
\item There exists a holomorphic orthogonal decomposition
$$
(E_{\infty},A_{\infty},h_{\infty})=\bigoplus_{i=1}^l(Q_{\infty,i},A_{\infty,i},h_{\infty,i})
$$
on $X\setminus (E_{nK}(\alpha)\cup Z_{\an})$ such that each $A_{\infty,i}$ is a $T$-admissible Hermitian-Einstein connection on $(Q_{\infty,i}, h_{\infty,i})$. Furthermore, 
\item on $X\setminus (E_{nK}(\alpha)\cup Z_{\an})$, there exists a parallel isomorphism of graded vector bundles on $X\setminus (E_{nK}(\alpha)\cup Z_{\an})$,
$$
g_{\infty}:\Gr^{\HNS}_{\alpha}(E,\delbar_E)=\bigoplus_{i=1}^l(Q_i,\delbar_{Q_i},h_i)\to \bigoplus_{i=1}^l(Q_{\infty,i},\delbar_{A_{\infty,i}},h_{\infty,i})= E_{\infty}
$$
where $h_i$ is a $T$-adapted Hermitian-Einstein metric on $(Q_i,\delbar_{Q_i})$. More precisely, the map $g_{\infty}$ induces an isomorphism $g_{\infty,i}:(Q_i,\delbar_{Q_i})\to (Q_{\infty,i},\delbar_{A_{\infty,i}})$ such that
$$
\nabla_{(\delbar_{A_{\infty,i}},h_{\infty,i})}\circ g_{\infty,i}=g_{\infty,i}\circ\nabla_{(\delbar_{Q_i},h_i)}
$$
where $\nabla_{(\delbar_{A_{\infty,i}},h_{\infty,i})}$ and $\nabla_{(\delbar_{Q_i},h_i)}$ denote the Chern connections of $h_{\infty,i}$ and $h_i$ on $(Q_{\infty,i},\delbar_{A_{\infty,i}})$ and $(Q_i,\delbar_{Q_i})$ respectively.
\end{itemize}
\end{theo}
\begin{rema}
\begin{enumerate}
\item By Theorem \ref{main thm}, we obtain that, if a holomorphic vector bundle $(E,\delbar_E)$ is $\alpha^{n-1}$-slope stable, then a $T$-Yang-Mills flow converges to a $T$-admissible Hermitian-Einstein connection which is equivalent to the Chern connection of a $T$-adapted Hermitian-Einstein metric of $(E,\delbar_E)$. It is a problem whether any $T$-admissible Hermitian-Einstein connection on an $\alpha^{n-1}$-slope stable bundle satisfies this property.
\item The properties that a cohomology class is nef and big, a reflexive sheaf is slope stable with respect to a nef and big class and a reflexive sheaf admits a $T$-adapted Hermitian-Einstein metric are birational invariant (refer to \cite{Jin25-1}, \cite{Jin26}). Hence our Theorem \ref{main thm2} and Theorem \ref{main thm} are generalized to compact normal Fujiki class varieties.
\end{enumerate}
\end{rema}
This paper is organized as follows. We fix 
a nef and big class $\alpha$. Without loss of generality, we can assume that the non-K\"{a}hler locus $D=E_{nK}(\alpha)$ is an snc divisor.
We fix an adapted current $T\in\alpha$ and its adapted approximation $\omega_{\ve}\in\alpha+\ve\omega_0$ (see Definition \ref{adapted defi}).

In Section \ref{preliminary}, we introduce basic definitions and preliminary results. In particular, the definitions of $\alpha^{n-1}$-Harder-Narasimhan-Seshadri filtrations, Yang-Mills flows, Uhlenbeck limit will be given in this section.

In Section \ref{YM sec}, we prove the solvability of the $T$-Yang-Mills flow (Proposition \ref{prop6}) and its convergence to a $T$-Yang-Mills connection $A_{\infty}$  on a reflexive sheaf $E_{\infty}$ defined over $X\setminus D$ in the sense of Uhlenbeck limit (Corollary \ref{cor10}). The $T$-Yang-Mills flow $A_t$ is constructed as a limit of $\omega_{\ve}$-Yang-Mills flow $A_{\ve,t}$ in $\ve\to0$. By the uniform estimates of the heat kernel of $(X,\omega_{\ve})$ established by Guo-Phong-Song-Sturm \cite{GPSS23}, the proof is similar to that of \cite{BS94}. Once we prove the locally smooth convergence of $A_{\ve,t}$ to $A_t$, the convergence of $A_t$ to a $T$-Yang-Mills connection $A_{\infty}$ follows from the argument of \cite{Sib15}. The convergent subsequence $A_{\ve_j,t_j}\to A_{\infty}$ with uniformly bounded curvatures will also be constructed (Lemma \ref{lem11}).
We remark that $E_{\infty}$ naturally admits a holomorphic orthogonal decomposition on $X\setminus D$ as in Corollary \ref{cor10}.
These will play important roles in this paper as explained below.

In Section \ref{HYM sec}, we prove that the $\alpha^{n-1}$-Harder-Narasimhan type $\mu_0=(\mu_1,\ldots,\mu_l)$ of the original holomorphic vector bundle $(E,\delbar_E)$ coincides with the eigenvalue $\mu_{\infty}=(\mu_{\infty,1},\ldots,\mu_{\infty,l\rq{}})$ of the curvature operator $\sqrt{-1}\Lambda_TF_{A_{\infty}}$. 
In \cite{DW04} and \cite{Sib15}, they proved this result by constructing approximate critical hermitian structures. In the setting of this paper, the degeneration of the adapted current makes it difficult to construct such a hermitian structure. We avoid this difficulty by using the approximating sequence $A_{\ve_j,t_j}$ constructed in section \ref{YM sec}.
First, we introduce the Hermitian-Yang-Mills type functional $\HYM_{\beta,N}^{\omega}$ for a K\"{a}hler metric $\omega$, and numbers $\beta\in\mathbb{R}_{\ge1}$ and $N\in\mathbb{R}_{\ge 0}$ following \cite{DW04} and \cite{Sib15}. 
Then we prove a key result Corollary \ref{cor15} which asserts the convergence of $\HYM_{\beta,N}^{\omega_{\ve_j}}(A_{\ve_j,t_j})$ to $\HYM_{\beta,N}(\mu_{\infty})$ in $j\to\infty$. As proven in \cite{DW04} and \cite{Sib15}, $\HYM_{\beta,N}^{\omega_{\ve}}(A_{\ve,t})$ is sufficiently close to $\HYM_{\beta,N}(\mu_{\ve})$ for large $t>t_{\ve}$ at each $\ve>0$. Here $\mu_{\ve}$ is the $\{\omega_{\ve}\}^{n-1}$-Harder-Narasimhan type of $(E,\delbar_E)$. Then, combining with the convergence of $\mu_{\ve}$ to $\mu_0$ in $\ve\to 0$ (Proposition \ref{prop3.1}) and the convergence of the Hermitian-Yang-Mills type functionals as above, we can prove $\mu_0=\mu_{\infty}$ (see Proposition \ref{prop16}).

In Section \ref{pf of main thm}, we prove the main Theorem \ref{main thm}. 
In section \ref{pf of main thm2}, following the approach of \cite{DW04} and \cite{Sib15}, we first refine the orthogonal decomposition of $E_{\infty}$ constructed in section \ref{YM sec} by making the $\alpha^{n-1}$-Harder-Narasimhan-Seshadri filtration of $E$ converge to a filtration of $E_{\infty}$ (Corollary \ref{cor18.0}, see also Proposition \ref{prop30.0}). This convergence is obtained by identifying each subsheaf with the corresponding projection operator and proving the convergence of these projection operators. A key difference from \cite{DW04} and \cite{Sib15} is that we analyze the projection operators along the sequence $A_{\ve_j,t_j}$ in section \ref{YM sec}, rather than along the $T$-YM flow $A_t$ constructed in section \ref{YM sec}. Indeed, since $A_t$ is singular, it is difficult to relate the degree of the subsheaf to that of its associated projection operator via the Chern-Weil formula (Lemma \ref{lem18}).
Then in section \ref{pf of main thm3}, we construct a parallel isomorphism from $\Gr^{\HNS}_{\alpha}(E)$, which is the graded sheaf associated to the $\alpha^{n-1}$-Harder-Narasimhan-Seshadri filtration of $E$, to $E_{\infty}$ with the holomorphic orthogonal decomposition constructed above (Theorem \ref{thm31}). In \cite{DW04} and \cite{Sib15}, this construction relies on the fact that each direct summand of $E_{\infty}$ defines a stable sheaf on $X$. In our setting, however, these summands are defined only on the complement of a divisor $D=E_{nK}(\alpha)$, and hence this argument cannot be directly applied. We overcome this difficulty by using the $T$-adapted Hermitian-Einstein metrics on the direct summands of $\Gr^{\HNS}_{\alpha}(E)$ and the $T$-admissible Hermitian-Einstein connections, induced by $A_{\infty}$, on the corresponding direct summands of $E_{\infty}$.

\section*{Acknowledgment}
The author thanks his supervisor Ryushi Goto for support.
The author thanks Prof. Masataka Iwai for giving him the opportunity to revisit the paper of Bando-Siu \cite{BS94}.
He is grateful to the organizers of \lq\lq{}New developments in Kobayashi-Hitchin correspondence and Higgs bundles 2\rq{}\rq{} for providing the opportunity to receive valuable discussions and comments on his work.
This work was supported by JSPS KAKENHI Grant Number JP26KJ1606.

\renewcommand{\thetheo}{\thesection.\arabic{theo}}
\section{Preliminary results and Definitions}\label{preliminary}
Throughout this paper, we denote the dimension of any compact K\"{a}hler manifold $X$ by $n$, and the rank of any complex vector bundle $E$ by $r$. 
\subsection{Positivities of cohomology classes}
Let $X$ be a compact K\"{a}hler manifold with a smooth K\"{a}hler metric $\omega$. 
In this paper, we will focus on the positivities of cohomology classes in the following sense.
\begin{defi}
\begin{enumerate}
\item A cohomology class $\alpha\in H^{1,1}(X,\mathbb{R})$ is big if $\alpha$ is represented by a K\"{a}hler current $T$, that is, there exists $\varepsilon>0$ such that $T\ge \varepsilon\omega$ in the sense of currents.
\item A cohomology class $\alpha$ is nef if for any $\varepsilon>0$ there exists a smooth $(1,1)$-form $\alpha_{\varepsilon}$ in $\alpha$ such that $\alpha_{\varepsilon}\ge -\varepsilon\omega$.
\end{enumerate}
\end{defi}
\noindent It is well-known that a nef cohomology class $\alpha$ is big if and only if its volume $\alpha^n$ is positive \cite{DP04}. 
\begin{defi}[\cite{Bou04}]
Let $\alpha$ be a big class on $X$. Then, the ample locus of $\alpha$ is defined as 
$$
\Amp(\alpha):=\{x\in X\mid \hbox{There is a K\"{a}hler current $T\in\alpha$ which is smooth K\"{a}hler around $x$}\}.
$$
The non-K\"{a}hler locus of $\alpha$ is the complement of the ample locus of $\alpha$, that is, $E_{nK}(\alpha):=X\setminus \Amp(\alpha)$.
\end{defi}
\noindent Boucksom \cite[Theorem 3.17]{Bou04} proved that the non-K\"{a}hler locus of a big class is a proper analytic subset of $X$.
The following result by Collins-Tosatti \cite{CT15} plays an important role:
\begin{theo}[\cite{CT15}]\label{CT}
Let $X$ be an $n$-dimensional compact K\"{a}hler manifold and $\alpha$ be a nef and big class. Let $D$ be a divisor contained in $E_{nK}(\alpha)$ the non-K\"{a}hler locus of $\alpha$. Then the following holds:
$$
\alpha^{n-1}[D]=0.
$$
\end{theo}
\subsection{Nonpluripolar product}
In this subsection, we introduce the notion of the nonpluripolar product which is introduced in \cite{BEGZ10}. We only state the special case. One refers to \cite{BEGZ10} for more general definition.
Let $T_i$ be a closed positive $(1,1)$-current on $X$ ($i=1,2$) contained in a big class $\alpha$. We fix a smooth representative $\theta$ of $\alpha$. Then, there exists a $\theta$-plurisubharmonic function $\varphi_i$ on $X$, called a potential of $T_i$, such that $T_i=\theta+dd^c\varphi_i$. Then we say that $T_1$ is less singular than $T_2$ if there is a bounded function $f$ on $X$ such that $\varphi_2\le \varphi_1+f$ holds. We say a closed positive $(1,1)$-current $T$ in a big class $\alpha$ has minimal singularities if $T$ is less singular than any other closed positive $(1,1)$-current in $\alpha$. A big class $\alpha$ contains a K\"{a}hler current on $X$ which is smooth K\"{a}hler on the ample locus $\Amp(\alpha)$ (\cite[Theorem 3.17]{Bou04}). Hence, any potential of a closed positive $(1,1)$-current $T$ with minimal singularities in a big class $\alpha$ is locally bounded on $\Amp(\alpha)$.   In particular, the wedge product $(T|_{\Amp(\alpha)})^p$ is well-defined closed positive $(p,p)$-current on $\Amp(\alpha)$ in the sense of Bedford-Taylor \cite{BT82}. Then the nonpluripolar product of a closed positive $(1,1)$-current $T$ with minimal singularities in a big class $\alpha$ is defined by the zero extension of $(T|_{\Amp(\alpha)})^p$ to $X$:
$$
\langle T^p\rangle =1_{\Amp(\alpha)}(T|_{\Amp(\alpha)})^p.
$$ 
If $T$ is a closed positive $(1,1)$-current with minimal singularities in a nef and big class $\alpha$, then $\alpha^p=\{\langle T^p\rangle\}$ for any $p>0$ \cite{BEGZ10}.

\subsection{Definition of \texorpdfstring{$\alpha^{n-1}$}{alpha{n-1}}-slope polystability}
In this section, we review the definition of slope polystability for nef and big classes which is introduced in \cite{Jin25-2} and its important properties.
\begin{defi}[\cite{Jin25-2}, see also \cite{Jin26}]\label{stability defi}
Let $X$ be a compact K\"{a}hler manifold and $\alpha$ be a nef and big class on $X$. Let $\mathcal{E}$ be a torsion free coherent sheaf on $X$. 
\begin{enumerate}
\item We define the $\alpha^{n-1}$-slope of $\mathcal{E}$, denoted by $\mu_{\alpha}(\mathcal{E})$, by 
$$
\mu_{\alpha}(\mathcal{E}):=\frac{1}{\rk \mathcal{E}}\int_Xc_1(\mathcal{E})\wedge\alpha^{n-1}.
$$
\item A torsion free sheaf $\mathcal{E}$ is $\alpha^{n-1}$-slope stable (resp.$\alpha^{n-1}$-slope semistable) if for any nontrivial torsion free subsheaf $\mathcal{F}\subset \mathcal{E}$ with $\rk\mathcal{F}<\rk \mathcal{E}$, the inequality 
$$
\mu_{\alpha}(\mathcal{F})< \mu_{\alpha}(\mathcal{E}) \hspace{2mm}(\hbox{ resp. } \mu_{\alpha}(\mathcal{F})\le \mu_{\alpha}(\mathcal{E}))
$$
holds.
\item A torsion free sheaf $\mathcal{E}$ is $\alpha^{n-1}$-slope polystable if there exists $\alpha^{n-1}$-slope stable torsion free sheaves $\mathcal{E}_1,\ldots,\mathcal{E}_k$ on $X$ with $\mu_{\alpha}(\mathcal{E}_1)=\cdots=\mu_{\alpha}(\mathcal{E}_k)$ such that there exists an isomorphism 
$$
\mathcal{E}|_{\Amp(\alpha)}\simeq (\mathcal{E}_1\oplus\cdots\oplus\mathcal{E}_k)|_{\Amp(\alpha)}
$$
as coherent sheaves on $\Amp(\alpha)$.
\end{enumerate}
\end{defi}   
For the analysis in this paper, the following is important.
\begin{lemm}[\cite{Jin25-1}]\label{open stable}
Let $X$ be a compact K\"{a}hler manifold with a K\"{a}hler class $\omega_0$ and $\alpha$ be a nef and big class on $X$. If a holomorphic vector bundle $E$ is $\alpha^{n-1}$-slope stable, then $E$ is $(\alpha+\varepsilon\omega_0)^{n-1}$-slope stable for sufficiently small $\varepsilon>0$.
\end{lemm}
\subsection{Adapted closed positive \texorpdfstring{$(1,1)$}{(1,1)}-currents}
In this subsection, we review the notion of adapted closed positive $(1,1)$-currents which is introduced by the author in \cite{Jin26}. We modify its definition. We remark that any nef and big class contains an adapted current (see Proposition \ref{adapted exists}).
\begin{defi}[cf.\cite{Jin26}]\label{adapted defi}
Let $X$ be a compact K\"{a}hler manifold and $\alpha$ be a nef and big class on $X$. Then, a closed positive $(1,1)$-current $T$ in $\alpha$ is said to be adapted if it satisfies the following conditions:
\begin{enumerate}
\item $T$ is smooth K\"{a}hler on $\Amp(\alpha)=X\setminus E_{nK}(\alpha)$.
\end{enumerate}
Let $\pi:Y\to X$ be a composition of blow-ups so that $D=E_{nK}(\pi^*\alpha)$ is an snc divisor. We denote by $s_D$ a defining section of $D$ and $h_D$ a smooth hermitian metric on $\mathcal{O}(D)$ such that $|s_D|_{h_D}\le 1$. Then
\begin{enumerate}
\setcounter{enumi}{1}
\item $\pi^*T\ge |s_D|^{2m}_{h_D}\omega_Y$ for some $m\in\mathbb{Z}_{\ge 0}$ where $\omega_Y$ is a smooth K\"{a}hler metric on $Y$.
\item There exists a sequence of smooth K\"{a}hler metrics $\omega_{\ve}$ lying in a K\"{a}hler class $\pi^*\alpha+\ve\omega_Y$ on $Y$ satisfying the following conditions:
\begin{itemize}
\item[{\rm (3-1)}] $\omega_{\ve}$ locally smoothly converges to $\pi^*T$ on $\Amp(\pi^*\alpha)=Y\setminus D$.
\item[{\rm (3-2)}] There exists $0<\delta<1$ such that
$$
|s_D|^{2-\delta}\frac{\omega_{\ve}^n}{\omega_Y^n}\le C.
$$
\item[{\rm (3-3)}] There exists $p>1$ such that
$$
\int_Y\left|\log\left(\frac{\omega_{\ve}^n}{\omega_Y^n}\right)\right|^p\omega_{\ve}^n\le C
$$
\end{itemize}
We call the above $\omega_{\ve}$ an adapted approximation of $T$.
\end{enumerate}
\end{defi}
\begin{rema}
\begin{enumerate}
\item By the condition (3-2) above, we see that an adapted current $T$ solves the complex Monge-Amp\`{e}re equation $\langle T^n\rangle=e^F\omega_0^n$ with $e^F\in L^p(X)$ for some $p>1$ and $\int_Xe^F\omega_0^n=\int_X\alpha^n$. Thus, by \cite{BEGZ10}, we obtain that $T$ has minimal singularities. 
\item If $\omega_{\ve}$ satisfies the conditions in Definition \ref{adapted defi} except for (3-3), and if $|\log(\omega_{\ve}^n/\omega_0^n)|$ is uniformly bounded by $-k\log|s_D|^2+C$ for some $k\ge 0$ and $C>0$, then the condition (3-3) is automatically satisfied.
\item A similar condition is considered in \cite[section 4.1]{Jin25-2}. One of the major differences is that an adapted current $T$ is not necessarily a K\"{a}hler current on $X$. Thus, we can only assume condition (1) above, instead of the stronger inequality $\pi^*T\ge \pi^*\omega_X$, which holds for a smooth K\"{a}hler current $T$ on $X$ where $\omega_X$ is a K\"{a}hler metric on $X$.
\end{enumerate}
\end{rema}
By the following proposition, we can always assume the existence of adapted currents in this paper.
\begin{prop}[\cite{Jin26} Proposition 2.17, see also \cite{BEGZ10} Theorem 5.1]\label{adapted exists}
Let $X$ be a compact K\"{a}hler manifold with a smooth K\"{a}hler metric $\omega_0$ on $X$ and $\alpha$ be a nef and big class on $X$. Let $T$ be a closed positive $(1,1)$-current in $\alpha$ defined by $\langle T^n\rangle =e^f\omega_0^n$ with $f\in C^{\infty}(X)$. Then $T$ is an adapted current. In particular, any nef and big class on a compact K\"{a}hler manifold contains an adapted current.
\end{prop}
\noindent As mentioned in \cite[Theorem 2.20]{Jin26}, the singular K\"{a}hler-Einstein metrics on compact klt K\"{a}hler varieties are examples of adapted currents (refer to the proof of \cite{CCHSTT25}).

The following analytic lemma plays an important role in this paper.
\begin{lemm}[\cite{GPSS23}, \cite{GPS24}]\label{sob ineq}
Let $\omega_{\ve}$ be an adapted approximation of an adapted current $T\in\alpha$ as in Definition \ref{adapted defi}. Then the following estimates hold.
\begin{enumerate}
\item {\rm (\cite[Theorem 2.1]{GPSS23})} There exists $q>1$ and $C>0$ such that for any $\ve$ and any $u\in L^2_1(X)$, we have
$$
\left(\int_X|u-\overline{u}|^{2q}\omega_{\ve}^n\right)^{1/q}\le C\int_X|\nabla u|^2_{\omega_{\ve}}\omega_{\ve}^n,
$$
where $\overline{u}=\int_Xu\omega_{\ve}^n/\int_X\omega_{\ve}^n$.
\item {\rm (\cite[Theorem 2.2]{GPSS23})} Let $q>1$ be the constant in (1). Let $H_{\omega_{\ve}}(x,y,t)$ be the heat kernel of $(X,\omega_{\ve})$. Then there exists $C>0$ and $t_0>0$ such that for any $\ve$ and $x,y\in X$,
$$
H_{\omega_{\ve}}(x,y,t)\le {C}t^{-\frac{q}{q-1}}\exc\left(-\frac{d_{\omega_{\ve}}(x,y)}{10t}\right), \hspace{4mm}\hbox{ if $t\in (0,t_0]$}
$$ 
and 
$$
H_{\omega_{\ve}}(x,y,t)\le {C}\exc\left(-\frac{d_{\omega_{\ve}}(x,y)}{10t}\right), \hspace{4mm}\hbox{ if $t\in (t_0,\infty)$}
$$
where $d_{\omega_{\ve}}(x,y)$ is the geodesic distance between $x$ and $y$ with respect to $\omega_{\ve}$.
\item {\rm (\cite[Lemma 2]{GPS24})} Let $v\in L^1(X,\omega_{\ve})$ and let $\Omega_a=\{v>-a\}$ for fixed constant $a>0$. Suppose that $v\in C^2(\overline{\Omega_a})$ and $\Delta_{\omega_{\ve}}v\ge -C\rq{}$ on $\Omega_a$ where $C\rq{}>0$ is a constant which is independent of $\ve$. Then there exists a constant $C>0$ which is independent of $\ve$ such that
$$
\sup_Xv\le C\left(1+\int_X|v|\omega_{\ve}^n\right).
$$
\end{enumerate}
\end{lemm}
\subsection{\texorpdfstring{$T$}{T}-adapted Hermitian-Einstein metrics}
For any connection $\nabla$ on a smooth vector bundle $E$, we denote by $F_{\nabla}:=\nabla\circ\nabla$ the curvature of $\nabla$. If $E$ is a holomorphic vector bundle, then for any hermitian metric $h$ on $E$, we denote by $\nabla_h$ the Chern connection of $h$.
Following \cite{Jin26}, we review the notion of adapted Hermitian-Einstein metrics. 
\begin{defi}[\cite{Jin26}]\label{adapted HYM defi}
Let $X$ be a compact K\"{a}hler manifold and $\alpha$ be a nef and big class on $X$. Let $T$ be a closed positive $(1,1)$-current in $\alpha$ which is smooth K\"{a}hler on $\Amp(\alpha)$. 
Let $\pi:Y\to X$ be a composition of blow-ups so that $D=E_{nK}(\pi^*\alpha)$ is an snc divisor. We denote by $s_D$ a defining section of $D$ and $h_D$ a smooth hermitian metric on $\mathcal{O}(D)$ such that $|s_D|_{h_D}\le 1$.
Let $(E, h_0)$ be a hermitian vector bundle on $X$ with a smooth hermitian metric $h_0$ on $E$.
\begin{enumerate}
\item A $T$-admissible Hermitian-Einstein (HE) connection on $(E,h_0)$ is a smooth hermitian connection $\nabla$ on $(E,h_0)|_{\Amp(\alpha)}$ which satisfies the following conditions:
\begin{enumerate}
\item (integrability) $\delbar^{\nabla}\circ\delbar^{\nabla}=0$.
\item (HE equation) $\sqrt{-1}\Lambda_TF_{\nabla}=\lambda\id_E$ on $\Amp(\alpha)$ for some constant $\lambda\in\mathbb{R}$.
\item (admissible condition) $\int_{\Amp(\alpha)}|F_{\nabla}|_{h_0,T}^2T^n<\infty$.
\end{enumerate}
We  call $\lambda$ in (b) the $T$-Hermitian-Einstein constant of $\nabla$.
\item If $E$ is a holomorphic vector bundle, then a $T$-admissible Hermitian-Einstein (HE) metric on $(E,\delbar_E)$ is a smooth hermitian metric $h$ on $E|_{\Amp(\alpha)}$ whose Chern connection $\nabla_h$ satisfies the following conditions (the integrability condition (a) above is automatic for the Chern connection):
\begin{itemize}
\item[(b)$^{\prime}$] $\sqrt{-1}\Lambda_TF_{\nabla_h}=\lambda\id_E$ on $\Amp(\alpha)$ for some constant $\lambda\in\mathbb{R}$.
\item[(c)$^{\prime}$]  $\int_{\Amp(\alpha)}|F_{\nabla_h}|_{h,T}^2T^n<\infty$.
\end{itemize}
We call $\lambda$ in (b)$^{\prime}$ the $T$-Hermitian-Einstein constant of $h$.
\item If $E$ is a holomorphic vector bundle, then a $T$-adapted Hermitian-Einstein (HE) metric on $(E,\delbar_E)$ is a $T$-admissible HE metric $h$ on $(E,\delbar_E)$ which additionally satisfies the following two conditions: If we denote by $h=h_0\Psi^2$ with a positive definite $h_0$-hermitian endomorphism $\Psi$, then it satisfies the following two conditions, called the $T$-adapted conditions.
\begin{enumerate}
\item[{(d)}] There exists $k\in\mathbb{Z}_{\ge 0}$ such that
$
\sup_{Y\setminus D}|s_D^k\pi^*\Psi|<\infty.
$
\item[{(e)}] There exists $l\in\mathbb{Z}_{\ge 0}$ such that
$
\int_{Y\setminus D}|s_D^l\partial^{\nabla_{\pi^*h_0}}(\pi^*\Psi)|^2_{\pi^*T}(\pi^*T)^n<\infty.
$
\end{enumerate}
\end{enumerate}
\end{defi}
\begin{rema}\label{HE const rema}
Let $T$ be an adapted current in a nef and big class $\alpha$ on a compact K\"{a}hler manifold $X$ and let $E$ be a holomorphic vector bundle on $X$. Then, by \cite{Jin26}, the $T$-HE constant $\lambda$ of a $T$-adapted HE metric on $E$ is given by
$$
\lambda=\frac{\mu_{\alpha}(E)}{\alpha^n/n!}.
$$
Since the RHS depends only on $E$ and $\alpha$, we also refer to $\lambda$ as the $T$-HE constant of $E$.
\end{rema}

\subsection{Kobayashi-Hitchin correspondence for nef and big classes}
The Kobayashi-Hitchin correspondence, classically, asserts that a holomorphic vector bundle $E$ over a compact K\"{a}hler manifold $(X,\omega)$ is $\{\omega\}^{n-1}$-slope polystable if and only if $E$ admits an $\omega$-HE metric. This is established by Donaldson \cite{Don85}, Uhlenbeck-Yau \cite{UY86}, Bando-Siu \cite{BS94} and Xuemiao Chen \cite{Chen25} when $\omega$ is a smooth K\"{a}hler metric.  There are several studies of the correspondence in the case that K\"{a}hler metrics have orbifold singularities (e.g. \cite{Faulk22}, \cite{CGNPPW23}, \cite{FO25}). Recently, the correspondence was generalized to more singular K\"{a}hler metrics such that the cohomology classes are nef and big by using the notions of adapted currents and adapted HE metrics (\cite{Jin26}, see also \cite{Jin25-2}):
\begin{theo}[\cite{Jin26}]\label{KH corr nef big}
Let $X$ be a compact K\"{a}hler manifold and $\alpha$ be a nef and big class on $X$. Let $T$ be an adapted closed positive $(1,1)$-current in $\alpha$. Let $E$ be a holomorphic vector bundle over $X$. Then, 
$E$ is $\alpha^{n-1}$-slope polystable if and only if $E$ admits a $T$-adapted HE metric.
Furthermore, if $E$ is $\alpha^{n-1}$-slope stable, then a $T$-adapted HE metric is unique up to scaling.
\end{theo}
The following estimate will be used later.
\begin{corr}[\cite{Jin26} Corollary 4.1]\label{est on HE}
Let $X$ be a compact K\"{a}hler manifold and $\alpha$ be a nef and big class on $X$ such that $D=E_{nK}(\alpha)$ is an snc divisor. Let $T$ be an adapted closed positive $(1,1)$-current in $\alpha$. Let $E$ be a holomorphic vector bundle over $X$. Let $s_D$ be a defining section of the snc divisor $D=E_{nK}(\alpha)$.
Let $h$ be a $T$-adapted HE metric on $E$. Then, for every $k\ge 0$, $s_D\cdot h^k$ is bounded as tensors. 
\end{corr}
The following generalization of the simpleness of stable bundles in the K\"{a}hler cases will also be used later:
\begin{lemm}[\cite{Jin26} Lemma 3.19]\label{simple stable}
Let $X$ be a compact K\"{a}hler manifold and $\alpha$ be a nef and big class on $X$ such that $D=E_{nK}(\alpha)$ is an snc divisor. Let $E$ be a holomorphic vector bundle over $X$. Let $s_D$ be a defining section of the snc divisor $D=E_{nK}(\alpha)$. Assume that $E$ is $\alpha^{n-1}$-slope stable. Then, for any holomorphic endomorphism $g\ne0$ of $E|_{X\setminus D}$ with 
$
\sup_{X\setminus D}|s_D^kg|<\infty
$
for some $k\ge 0$, there exists a constant $0\ne a\in \mathbb{C}$ such that $g=a\id_E$. In particular, if $\alpha^{n-1}$-slope stable vector bundle $E$ is written as $E|_{X\setminus D}=(\mathcal{F}\oplus\mathcal{G})|_{X\setminus D}$ for some reflexive sheaves $\mathcal{F}$ and $\mathcal{G}$ on $X$, then $\mathcal{F}=0$ or $\mathcal{G}=0$.
\end{lemm}

\subsection{The $\alpha^{n-1}$-Harder-Narasimhan-Seshadri filtrations}
In this section, we introduce the notion of $\alpha^{n-1}$-Harder-Narasimhan-Seshadri filtrations of holomorphic vector bundles where $\alpha$ is a nef and big class.
Let $X$ be a compact K\"{a}hler manifold with a K\"{a}hler class $\omega_0$, $\alpha$ be a nef and big class  and $E$ be a holomorphic vector bundle on $X$.
The following result holds:
\begin{lemm}[see \cite{Jin25-1}]\label{bound deg}
There exists a saturated subsheaf $\mathcal{F}\subset E$ such that 
$$
\mu_{\alpha}(\mathcal{F})=\sup\{\mu_{\alpha}(\mathcal{F}\rq{})\mid \hbox{$\mathcal{F}\rq{}\subset E$ is a nonzero proper saturated subsheaf}\}
$$
\end{lemm}
As a consequence, we can prove the existence of the Harder-Narasimhan-Seshadri filtration. We remark that the uniqueness is obtained only on the complement of the non-K\"{a}hler locus, since $\alpha^{n-1}\cdot[D]=0$ for any divisor contained in the non-K\"{a}hler locus by \cite{CT15} (refer to Theorem \ref{CT}).
\begin{prop}\label{HN def}
Let $X$ be a compact K\"{a}hler manifold and $\alpha$ be a nef and big class on $X$. Let $E$ be a holomorphic vector bundle on $X$. 
\begin{enumerate}
\item There exists a sequence of saturated subsheaves
$$
0=E_0\subset E_1 \subset E_2\subset\cdots\subset E_l=E
$$
such that $Q_i:=E_i/E_{i-1}$ is $\alpha^{n-1}$-slope semistable and $\mu_{\alpha}(Q_i)>\mu_{\alpha}(Q_{i+1})$. Any two such filtrations are isomorphic on the ample locus $\Amp(\alpha)=X\setminus E_{nK}(\alpha)$. We call the above filtration by the $\alpha^{n-1}$-Harder-Narasimhan (HN) filtration. We denote by 
$
\Gr^{HN}_{\alpha}(E)=\bigoplus_{i=1}^lQ_i
$
the associated graded object.
\item Suppose that $E$ is $\alpha^{n-1}$-slope semistable. Then there exists a filtration
$$
0=E_0\subset E_1 \subset E_2\subset\cdots\subset E_l=E
$$
such that $Q_i:=E_i/E_{i-1}$ is $\alpha^{n-1}$-slope stable and $\mu_{\alpha}(Q_i)=\mu_{\alpha}(E)$. We call the above filtration by an $\alpha^{n-1}$-Seshadri filtration.  We denote by 
$
\Gr^{S}_{\alpha}(E)=\bigoplus_{i=1}^lQ_i
$
the associated graded object. Any two such graded objects $\Gr^{S}_{\alpha}(E)$ associated to $\alpha^{n-1}$-Seshadri filtrations are isomorphic on the ample locus $\Amp(\alpha)=X\setminus E_{nK}(\alpha)$.
\end{enumerate}
\end{prop}
Combining these two filtrations, we obtain the following:
\begin{defi}\label{HNS def2}
Let $X$ be a compact K\"{a}hler manifold and $\alpha$ be a nef and big class on $X$. Let $E$ be a holomorphic vector bundle on $X$. Let us consider the $\alpha^{n-1}$-HN filtration
$$
0=E_0\subset E_1 \subset E_2\subset\cdots\subset E_l=E.
$$
Then there exists filtrations
$$
E_{i-1}=E_{i,0}\subset E_{i,1}\subset \cdots \subset E_{i,k_i}=E_i
$$
such that each $Q_{i,j}=E_{i,j}/E_{i,j-1}$ is $\alpha^{n-1}$-slope stable and $\mu_{\alpha}(Q_{i,j})=\mu_{\alpha}(E_i)$. We call the filtration $\{E_{i,j}\}$ of $E$ by an $\alpha^{n-1}$-Harder-Narasimhan-Seshadri (HNS)  filtration of $E$. We denote by 
$
\Gr_{\alpha}^{\HNS}(E)=\bigoplus_{i,j}Q_{i,j}
$
the associated graded object. All the graded objects $\Gr^{\HNS}_{\alpha}(E)$ associated to   $\alpha^{n-1}$-HNS filtrations are isomorphic when restricted to the ample locus $\Amp(\alpha)=X\setminus E_{nK}(\alpha)$.
\end{defi}
By the uniqueness of the graded objects associated to the $\alpha^{n-1}$-HNS filtrations together with Theorem \ref{CT}, we can define as follows:
\begin{defi}\label{HNS def}
Let $X$ be a compact K\"{a}hler manifold and $\alpha$ be a nef and big class on $X$. Let $E$ be a holomorphic vector bundle on $X$. Let us consider the $\alpha^{n-1}$-HNS filtration
\begin{align}\label{HNS}
&0=E_{1,0}\subset E_{1,1} \subset E_{1,2}\subset\cdots\subset E_{1,k_1}=E_1\notag\\
&\hspace{2mm}\cdots\notag\\
&\hspace{2mm}=E_{i,0}\subset E_{i,1}\subset \cdots\subset E_{i,k_i}=E_i\notag\\
&\hspace{2mm}\cdots\notag\\
&\hspace{2mm}=E_{l,0}\subset E_{l,1}\subset\cdots\subset  E_{l,k_l}=E
\end{align}
Then, the $\alpha^{n-1}$-Harder-Narasimhan (HN) type of $E$ is defined by
$$
\mu_0=(\mu_1,\ldots,\mu_1,\ldots,\mu_l,\ldots,\mu_l)
$$
where each $\mu_i=\mu_{\alpha}(E_i/E_{i-1})=\mu_{\alpha}(E_{i,j}/E_{i,j-1})$ is repeated $k_i$ times.
\end{defi}
\noindent Then the following proposition is proved in the same way as \cite[Proposition 4.12]{Sib15}.
\begin{prop}[cf. \cite{Sib15} Proposition 4.12]\label{prop3.1}
Let $X$ be a compact K\"{a}hler manifold with a K\"{a}hler class $\omega_0$ and $\alpha$ be a nef and big class on $X$. Let $E$ be a holomorphic vector bundle on $X$. We denote by $\omega_{\ve}=\alpha+\ve\omega_0$ a sequence of K\"{a}hler classes.
Let $\mu_{\ve}$ be the $\omega_{\ve}^{n-1}$-HN type of $E$ and $\mu_0$ be the $\alpha^{n-1}$-HN type of $E$. Then we have $\mu_{\ve}\xrightarrow{\ve\to0} \mu_0$.
\end{prop}
As explained in \cite[section 4]{Sib15}, an $\alpha^{n-1}$-HNS filtration (\ref{HNS}) of a holomorphic vector bundle $E$ corresponds to a rational section of a suitable flag bundle of $E$. Hence, if we pull back the flag bundle along a suitable birational morphism $\pi$, the rational section extends to a holomorphic section which defines a filtration $(\wtil{E}_{i,j})_{i,j}$ of $\pi^*E$ such that each $\wtil{E}_{i,j}\subset \pi^*E$ is a holomorphic subbundle. Since each $\wtil{E}_{i,j}$  is isomorphic to $\pi^*E_{i,j}$ away from the exceptional divisor, we can see that $(\wtil{E}_{i,j})_{i,j}$ defines a $\pi^*\alpha^{n-1}$-HNS filtration of $\pi^*E$. Hence we obtain the following (refer to \cite[section 4]{Sib15} for more detailed argument).
\begin{prop}[cf. \cite{Sib15} section 4]\label{HNS resol prop}
Let $X$ be a compact K\"{a}hler manifold and $\alpha$ be a nef and big class on $X$. Let $E$ be a holomorphic vector bundle on $X$. Let $(E_{i,j})_{i,j}$ be an $\alpha^{n-1}$-HNS filtration of $E$ as (\ref{HNS}). Then there exists a sequence of blow ups $\pi:\wtil{X}\to X$ and a $\pi^*\alpha^{n-1}$-HNS filtration $(\wtil{E}_{i,j})_{i,j}$ of $\pi^*E$ such that each $\wtil{E}_{i,j}$ is a holomorphic subbundle of $\pi^*E$ and $\wtil{E}_{i,j}=\pi^*E_{i,j}$ away from the exceptional divisor. In particular, the $\pi^*\alpha^{n-1}$-HN type of $E$ equals the $\alpha^{n-1}$-HN type of $E$.
\end{prop}
\noindent Therefore, we can assume that the $\alpha^{n-1}$-HNS filtration is formed by subbundles. 
\subsection{Weakly holomorphic projections}
Uhlenbeck-Yau \cite{UY86} defined the notion of a weakly holomorphic projection of a holomorphic hermitian vector bundle which is a projection operator whose image defines a coherent analytic subsheaf.
\begin{defi}[\cite{UY86}]
Let $(E,\delbar_E,h)$ be a holomorphic hermitian vector bundle over a K\"{a}hler manifold $(X,\omega)$. Then, a weakly holomorphic projection is an endomorphism $\pi\in L^2_{1,\loc}(\End(E))$ satisfying
$\pi^2=\pi=\pi^*$ and $(\id_E-\pi)\circ\delbar_E\pi=0$ 
almost everywhere. 
We define its degree as follows when it is well-defined: 
$$
\deg_{\omega}(\pi):=\int_X\Tr(\sqrt{-1}\Lambda_{\omega}F_{\nabla_h}\circ\pi)\omega^n
-\int_X|\delbar_E\pi|^2_{h,\omega}\omega^n.
$$
\end{defi}
Uhlenbeck-Yau proved the following.
\begin{prop}[\cite{UY86}]
Let $(E,\delbar_E,h)$ be a holomorphic hermitian vector bundle over a K\"{a}hler manifold $(X,\omega)$.  Let $\pi\in L^2_{1,\loc}(\End(E))$ be a weakly holomorphic projection of $(E,\delbar_E,h)$. Then, there exists a coherent subsheaf $\mathcal{F}\subset (E,\delbar_E)$ with a singular set $V\subset X$ such that
$\codim_{\mathbb{C}}V\ge 2$,
$\pi|_{X\setminus V}$ is $C^{\infty}$ and
$\mathcal{F}|_{X\setminus V}=(\pi|_{X\setminus V})(E|_{X\setminus V})\subset E|_{X\setminus V}$ is a holomorphic subbundle.
\end{prop}
\noindent If $X$ is compact, then the degree of a weakly holomorphic projection coincides with the $\{\omega\}^{n-1}$-degree of the corresponding coherent sheaf.

\subsection{Yang-Mills flows and Uhlenbeck compactness}
\subsubsection{Definitions of the Yang-Mills flows and the Yang-Mills connections}\label{YM def sec}
We review some basic definitions around Yang-Mills connections following \cite{Sib15}.
Let $X$ be a K\"{a}hler manifold with a smooth K\"{a}hler metric $\omega$. Let $E$ be a complex vector bundle on $X$ with a smooth hermitian metric $h_0$. Then, a holomorphic structure on $E$ is identified with a $\mathbb{C}$-linear map $\delbar_E: E\to E\otimes T^{0,1}_X$ satisfying the Leibniz rule and the integrability condition $\delbar_E\circ \delbar_E=0$.
Abstractly, a connection $\nabla$ is said to be integrable if its $(0,1)$-part $\delbar^{\nabla}:=\nabla^{0,1}$ satisfies the integrability condition $\delbar^{\nabla}\circ\delbar^{\nabla}=0$. We will represent a connection either by its covariant derivative $\nabla$ or by its connection 1-form $A$. We will also denote by $\nabla_A$ when its connection 1-form is $A$. 
 We denote the Chern connection of a hermitian metric $h_0$ on a holomorphic vector bundle $(E,\delbar_E)$ by $\nabla_{(\delbar_E,h_0)}$, $(\delbar_E,h_0)$ or simply $\nabla_h$.

Let $\mathcal{A}^{1,1}_{h_0}$ be the set of integrable $h_0$-unitary connections. 
Then the Yang-Mills functional is a functional $\YM(\cdot)$ on $\mathcal{A}^{1,1}_{h_0}$ defined by
$$
\YM(\nabla):=\int_X|F_{\nabla}|^2_{h_0}\frac{\omega^n}{n!},
$$
where $F_{\nabla}=\nabla\circ\nabla$ is the curvature tensor of $\nabla$.
Let us denote  by $\mathcal{G}:= U(E,h_0)$ the gauge group of $(E,h_0)$ and  by $\mathcal{G}^{\mathbb{C}}$ the complex-gauge group. Then $\mathcal{G}$ and $\mathcal{G}^{\mathbb{C}}$ act on $\mathcal{A}^{1,1}_{h_0}$ in the following way: for $g\in \mathcal{G}$ (resp. $g\in\mathcal{G}^{\mathbb{C}}$) and $\nabla\in\mathcal{A}^{1,1}_{h_0}$,
$$
g\cdot \nabla:=g\circ\nabla\circ g^{-1}.
$$
If a connection 1-form is $A$, then we write the induced action on the connection 1-form by $g\cdot A$ or $g(A)$.
By definition of actions above, we see that the Yang-Mills functional descends to 
$$
\YM:\mathcal{A}^{1,1}_{h_0}/\mathcal{G}\to\mathbb{R}.
$$
On the other hand, there is another natural functional on $\mathcal{A}^{1,1}_{h_0}/\mathcal{G}$, called the Hermitian-Yang-Mills functional $\HYM(\cdot)$ defined as
$$
\HYM(\nabla):=\int_X|\Lambda_{\omega}F_{\nabla}|_{h_0}^2\frac{\omega^n}{n!}.
$$
If $X$ is compact, these two functionals $\YM(\cdot)$ and $\HYM(\cdot)$ are related as
\begin{equation}\label{eq0}
\YM(\nabla)-\HYM(\nabla)=c_n(2c_2(E)-c_1(E)^2)\cdot\{\omega\}^{n-2},
\end{equation}
where $c_n$ is a constant depending only on $n=\dim X$. Therefore, $\YM(\cdot)$ and $\HYM(\cdot)$ have the same critical points and gradient flows. A critical point of $\YM(\cdot)$ is said to be an $\omega$-Yang-Mills (YM) connection and a gradient flow of $\YM(\cdot)$ is said to be an $\omega$-Yang-Mills (YM) flow. More explicitly, an $\omega$-YM connection is an integrable $h_0$-unitary connection $\nabla=\nabla_A$, where $A$ is a connection 1-form of $\nabla$, satisfying 
$$
d_{A}^{*_{h_0}}F_{A}=0.
$$ 
Using the K\"{a}hler identity, we can see that a connection $A$ is $\omega$-YM if and only if it is an $\omega$-HYM connection:
$$
d_A(\Lambda_{\omega}F_A)=0.
$$
An $\omega$-YM flow with initial condition $\nabla_{A_0}\in\mathcal{A}^{1,1}_{h_0}$ is a family of integrable $h_0$-unitary connections $\nabla_{A_t}$ such that $\nabla_{A_{t=0}}=\nabla_{A_0}$ and
$$
\frac{\partial{A_t}}{\partial t}=-d_{A_t}^*F_{A_t}.
$$
\subsubsection{The Hermitian-Yang-Mills flows and the Uhlenbeck compactness theorem}
It follows from \cite{Don85} that,
on a holomorphic vector bundle over a K\"{a}hler manifold $(X,\omega)$, an $\omega$-YM flow is complex-gauge equivalent to a parabolic PDE of hermitian metrics, called an $\omega$-Hermitian-Yang-Mills (HYM) flow (see Lemma \ref{lem0}). By \cite{Don85} and \cite{Sim88}, an $\omega$-HYM flow, thus the corresponding $\omega$-YM flow, has a smooth solution for all time. Moreover, any $\omega$-YM flow on $\mathcal{A}^{1,1}_{h_0}$ preserves $\mathcal{G}^{\mathbb{C}}$-orbits (see \cite{Don85}).
On a hermitian vector bundle $(E,h_0)$, we define an action of $g\in\mathcal{G}^{\mathbb{C}}$ to $h_0$ by 
$$
g\cdot h_0:=h_0\cdot g^*g=h_0\cdot g^2,
$$
that is, we define $(g\cdot h_0)(u,v):=h_0(g(u), g(v))$ for any local sections $u$ and $v$.
If $g\in\mathcal{G}$, then we obviously have $g\cdot h_0=h_0$.
\begin{lemm}[see \cite{Wil08} section 3.1]\label{lem0}
Let $(E,h_0)$ be a smooth hermitian vector bundle with an integrable unitary connection $\nabla_{A_0}$ over a K\"{a}hler manifold $(X,\omega)$. Then the following holds.
\begin{enumerate}
\item Let $\nabla_{A_t}=g_t\circ\nabla_{A_0}\circ g_t^{-1}$ be an $\omega$-YM flow on $(E,h_0)$ with initial condition $\nabla_{A_0}$ where $g_t\in\mathcal{G}^{\mathbb{C}}$. Then $h_t:=h_0g_t^2$ is an $\omega$-HYM flow on $(E,\delbar_{A_0})$ with initial condition $h_0$:
$$
h_t^{-1}\frac{dh_t}{dt}=-(\sqrt{-1}\Lambda_{\omega}F_{h_t}-\lambda\id), \hspace{4mm} h_{t=0}=h_0,
$$
where $\lambda$ is the $\omega$-HE constant of $E$.
\item Let $h_t=h_0g_t^2$ be an $\omega$-HYM flow on $(E,\delbar_{A_0})$ with initial condition $h_0$. Then there exists a real gauge equivalence $S_t\in\mathcal{G}$ such that $\nabla_{A_t}=(S_t\circ g_t)\circ \nabla_{(\delbar_{A_0},h_t)}\circ(S_t\circ g_t)^{-1}$ is an $\omega$-YM flow on $(E,h_0)$ with initial condition $\nabla_{A_0}$.
\end{enumerate}
\end{lemm}
The following estimates of curvatures along the flows will be used later.
\begin{lemm}[see e.g. \cite{Kob87} Chapter6, lemma 8.7 and its proof]\label{lem1}
Let $(X,\omega)$ be a compact K\"{a}hler manifold and $(E,h_0)\to X$ be a hermitian vector bundle with an integrable unitary connection $A_0$. Let $A_t$ be the $\omega$-YM flow on $(E,h_0)$ with initial condition $A_0$. Then the curvature $F_{A_t}$ satisfies the following:
\begin{enumerate}
\item We have
$$
\frac{d}{dt}\|F_{A_t}\|_{L^2(\omega)}^2=-2\|d_{A_t}^*F_{A_t}\|^2_{L^2(\omega)}\le 0.
$$
In particular, $t\mapsto \YM(A_t)$ and $t\mapsto \HYM(A_t)$ are non-increasing.
\item We have
$$
\left(\frac{d}{dt}-\Delta_{\omega}\right)|\Lambda_{\omega}F_{A_t}|^2\le 0.
$$
Hence, by the maximum principle $t\mapsto \|\Lambda_{\omega}F_{A_t}\|_{L^{\infty}}$ is non-increasing.
\item We have
$$
\left(\frac{d}{dt}-\Delta_{\omega}\right)|\Lambda_{\omega}F_{A_t}|\le 0,
$$
Hence we have
$$
\int_X|\Lambda_{\omega}F_{A_t}|\omega^n\le \int_X|\Lambda_{\omega}F_{A_0}|\omega^n.
$$
\end{enumerate}
\end{lemm}
Above Lemma \ref{lem1} implies that we can apply the following Uhlenbeck compactness theorem for $A_t$. We will need more general statement as follows:
\begin{theo}[\cite{UY86} Theorem 5.2, \cite{Uhl}, \cite{Nak88}, see also \cite{Sib15} Theorem 2.17]\label{thm2}
Let $(X,\omega)$ be a (not necessarily compact) K\"{a}hler manifold and $(E,h_0)\to X$ be a hermitian vector bundle on $X$. Fix any $p>n$. Let ${A_j}$ be a sequence of integrable unitary connections on $(E,h_0)$ such that $\|\Lambda_{\omega}F_{A_j}\|_{L^{\infty}(X)}$ and $\|F_{A_j}\|_{L^2(X)}$ are uniformly bounded. Then, there exists
a subsequence, still denoted by $A_j$, a closed subset $Z_{\an}\subset X$ of Hausdorff codimension at least 4,  a smooth hermitian vector bundle $(E_{\infty},h_{\infty})\to X\setminus Z_{\an}$ with an $L^p_{1,\loc}(X\setminus Z_{\an})$ integrable unitary connection $A_{\infty}$ and 
 a sequence of isometries $u_j:(E, h_0)\to (E_{\infty},h_{\infty})$ on $X\setminus Z_{\an}$
such that $u_j\cdot (A_j|_{X\setminus Z_{\an}}) \to A_{\infty}$ weakly in $L^p_{1,\loc}(X\setminus Z_{\an})$.
\end{theo}
We will call the limit $A_{\infty}$ above a Uhlenbeck limit. Hereafter, we will omit $u_j$. The following gives a sufficient condition so that a Uhlenbeck limit is Yang-Mills:
\begin{corr}[see \cite{DW04} Proposition 2.11 and Corollary 2.12]\label{cor3}
In addition to the assumptions in Theorem \ref{thm2}, we also require that
$$
\|d_{A_j}\Lambda_{\omega}F_{A_j}\|_{L^2(X)}\to 0.
$$
Then the following holds:
\begin{enumerate}
\item Any Uhlenbeck limit $A_{\infty}$ of $A_j$ is an $\omega$-YM connection on $E_{\infty}$. Therefore, on $X\setminus Z_{\an}$, there is a natural holomorphic orthogonal splitting
\begin{equation}\label{eq2}
(E_{\infty},A_{\infty},h_{\infty})=\bigoplus_{i=1}^l(Q_{\infty,i},A_{\infty,i},h_{\infty,i}).
\end{equation}
\item For all $1\le p <\infty$, we have $\Lambda_{\omega}F_{A_j}\to \Lambda_{\omega}F_{A\infty}$ strongly in $L^p(X\setminus Z_{\an})$.
\end{enumerate}
\end{corr}
Since $A_{\infty}$ is $\omega$-YM, it satisfies $d_{A_{\infty}}\Lambda_{\omega}F_{A_{\infty}}=0$. It implies that the eigenvalues $\mu_{\infty}=(\mu_{\infty,1},\ldots, \mu_{\infty,l})$ of $\sqrt{-1}\Lambda_{\omega}F_{A_{\infty}}$ are constants. Then the splitting (\ref{eq2}) is the eigenvalue decomposition of $E_{\infty}$ with respect to $\sqrt{-1}\Lambda_{\omega}F_{A_{\infty}}$. 

\section{Proof of Theorem \ref{main thm2}}\label{YM sec}
In this section, we fix a holomorphic hermitian vector bundle $(E,\delbar_E,h_0)$ over a compact K\"{a}hler manifold $X$ with a smooth K\"{a}hler metric $\omega_0$. We also fix a nef and big class $\alpha$ on $X$ such that $D=E_{nK}(\alpha)$ is an snc divisor. Let $s_D$ be a defining section of $D$. We fix an adapted current $T\in\alpha$ and its adapted approximation $\omega_{\ve}\in \alpha+\ve\omega_0$.
\subsection{Existence of the $T$-YM flows}\label{exist YM}
Since the YM flow corresponds to the HYM flow as in Lemma \ref{lem0}, we first consider the $\omega_{\ve}$-HYM flow. Let $(h_{\ve,t})_{t\in[0,\infty)}$ be the $\omega_{\ve}$-HYM flow on $(E,\delbar_E)$ with $h_{\ve,t=0}=h_0$:
\begin{equation}\label{eq4}
h_{\ve,t}^{-1}\frac{dh_{\ve,t}}{dt}=-(\sqrt{-1}\Lambda_{\omega_{\ve}}F_{(\delbar_E,h_{\ve,t})}-\lambda_{\ve}\id), \hspace{ 4mm} h_{\ve,t=0}=h_0
\end{equation}
where $\lambda_{\ve}$ is the $\omega_{\ve}$-HE constant of $E$. We have the following uniform estimates for curvature of $h_{\ve,t}$.
\begin{lemm}[see \cite{Sib15} section 6, see also \cite{BS94}]\label{lem4}
The following holds:
\begin{enumerate}
\item There is a constant $C>0$ such that for any $\ve>0$ and $t>0$,
$$
\int_X|\Lambda_{\omega_{\ve}}F_{(\delbar_E,h_{\ve,t})}|_{h_{\ve,t}}\omega_{\ve}^n\le \int_X|\Lambda_{\omega_{\ve}}F_{(\delbar_E,h_0)}|_{h_0}\omega_{\ve}^n \le C.
$$
\item There exists $t_0>0$ and $C=C(t_0)>0$ such that for any $\ve>0$ and $t\ge t_0$,
$$
\|\Lambda_{\omega_{\ve}}F_{(\delbar_E,h_{\ve,t})}\|_{L^{\infty}(X,h_{\ve,t})}\le C \hbox{ and } \int_X|F_{(\delbar_E,h_{\ve,t})}|_{h_{\ve,t}}^2\omega_{\ve}^n\le C. 
$$
\item For any $K\Subset X\setminus D$, there exists $C_K>0$ such that for any $t\ge 0$,
$$
\|\Lambda_{\omega_{\ve}}F_{(\delbar_E,h_{\ve,t})}\|_{L^{\infty}(K)}\le C_K.
$$
Here $C_K>0$ depends on $K$, $\sup_{\ve>0}\|\Lambda_{\omega_{\ve}}F_{(\delbar_E,h_0)}\|_{L^{\infty}(K)}$ and $\sup_{\ve>0}\|\Lambda_{\omega_{\ve}}F_{(\delbar_E,h_0)}\|_{L^1(X,\omega_{\ve})}$. 
\end{enumerate}
\end{lemm}
\begin{proof}
(1) Since $h_0$ is smooth on $X$, there exists $C>0$ such that $-C\omega_0\cdot\id_E\le \sqrt{-1}F_{(\delbar_E,h_0)}\le C\omega_0\cdot\id_E$. Hence,
\begin{align*}
|\Lambda_{\omega_{\ve}}F_{(\delbar_E,h_0)}|_{h_0}\omega_{\ve}^n
&\le \left|\Lambda_{\omega_{\ve}}(\sqrt{-1}F_{(\delbar_E,h_0)}-C\omega_0\cdot\id_E)\right|\omega_{\ve}^n + \left|\Lambda_{\omega_{\ve}}(C\omega_0\cdot\id_E)\right|\omega_{\ve}^n\\
&\le \Tr\left(2C\omega_0\cdot \id_E-\sqrt{-1}F_{h_0}\right)\wedge\omega_{\ve}^{n-1}.
\end{align*}
Hence its integral over $X$ gives
$$
\int_X|\Lambda_{\omega_{\ve}}F_{(\delbar_E,h_0)}|_{h_0}\omega_{\ve}^n
\le (2rC\{\omega_0\}-c_1(E))\cdot(\alpha+\ve\omega_0)^{n-1}\le C.
$$
Then the first inequality in (1) is a consequence of Lemma \ref{lem1} (3).\\
Since $\omega_{\ve}$ is an adapted approximation, we can use the uniform heat kernel estimates of $(X,\omega_{\ve})$ in Lemma \ref{sob ineq} (2). Together with Lemma \ref{lem1}, we obtain (2) and (3). 
\end{proof}
In particular, the uniform $L^{\infty}_{\loc}(X\setminus D)$ estimate in Lemma \ref{lem4} gives the following uniformly local estimates on $h_{\ve,t}$. If we further use Lemma \ref{sob ineq} and Lemma \ref{lem4}, we can prove the following lemma in the same way as in \cite{BS94} (see also \cite[section 6]{Sib15}).
\begin{lemm}[see \cite{BS94} and \cite{Sib15} section 6]\label{lem5}
For any $K\Subset X\setminus D$ and $0<t<T$, there exists $C_{K,T}>0$ such that
\begin{enumerate}
\item $\|h_{\ve,t}^{\pm1}\|_{L^{\infty}(K)}\le C_{K,T},$
\item $\|h_{\ve,t}\|_{C^{1,\beta}(K)}\le C_{K,T}$ for any $0\le \beta<1$.
\end{enumerate}
Furthermore, the constant $C_{K,T}$ depends only on $K$, $T$,  $\sup_{\ve>0}\|\Lambda_{\omega_{\ve}}F_{(\delbar_E,h_0)}\|_{L^{\infty}(K)}$ and \\ $\sup_{\ve>0}\|\Lambda_{\omega_{\ve}}F_{(\delbar_E,h_0)}\|_{L^1(X,\omega_{\ve})}$. 
\end{lemm}
Then we are ready to prove the convergence of $h_{\ve,t}$ in $\ve\to 0$.
\begin{prop}[cf. \cite{Sib15} section 6]\label{prop6}
Let $h_{\ve,t}$ be the $\omega_{\ve}$-HYM flow on $(E,\delbar_E)$ with initial condition $h_0$ in (\ref{eq4}). Then the following holds:
\begin{enumerate}
\item There exists a subsequence of $\ve$, also denoted by $\ve$, and a smooth hermitian metric $h_t$ on $E|_{X\setminus D}$ such that $h_{\ve,t}\to h_t$ strongly in $C^{\infty}_{\loc}(X\setminus D\times [0,\infty))$.
\item Then $h_t$ above is a $T$-HYM flow on $(E,\delbar_E)|_{X\setminus D}$ with initial condition $h_0$:
$$
h_t^{-1}\frac{dh_t}{dt}=-(\sqrt{-1}\Lambda_TF_{(\delbar_E,h_t)}-\lambda\id_E) \hbox{,  } h_{t=0}=h_0.
$$
Here $\lambda$ is the $T$-HE constant of $E$ (refer to Remark \ref{HE const rema}).
\end{enumerate}
\end{prop}
\begin{proof}
(1) If we write the curvature locally, we obtain
\begin{equation}\label{eq7}
\Delta_{\omega_{\ve}}h_{\ve,t}=h_{\ve,t}\sqrt{-1}\Lambda_{\omega_{\ve}}F_{(\delbar_E,h_{\ve,t})}+\sqrt{-1}\Lambda_{\omega_{\ve}}\delbar h_{\ve,t}\cdot h_{\ve,t}^{-1}\cdot \partial h_{\ve,t}=:f_{\ve,t}.
\end{equation}
Then, by Lemma \ref{lem4} (3), by Lemma \ref{lem5}, and by the locally smooth convergence of $\omega_{\ve}$ to a K\"{a}hler metric $T|_{X\setminus D}$ on $X\setminus D$ together with a diagonalization procedure, the result follows from the standard elliptic estimate (e.g. \cite[Theorem 9.11]{GT01}) and the proof in \cite[Chapter 6, Theorem 7.1]{Kob87}.
 By Lemma \ref{lem5} and (1), $h_{\ve,t}^{-1}$ also locally smoothly converges to $h_t^{-1}$ on $X\setminus D\times[0,\infty)$. Then, since $\omega_{\ve}$ locally smoothly converges to a K\"{a}hler metric $T$ on $X\setminus D$, we have that $h_t$ satisfies the $T$-HYM flow equation, that is, we obtain (2).
\end{proof}
By Lemma \ref{lem0}, we obtain the following:
\begin{corr}\label{cor7}
Let $h_{\ve,t}=h_0g_{\ve,t}^2$ be the $\omega_{\ve}$-HYM flow and $h_t=h_0g_t^2$ be the $T$-HYM flow in the previous Proposition \ref{prop6}. Then, the corresponding $\omega_{\ve}$-YM flow $A_{\ve,t}=g_{\ve,t}\cdot (\delbar_E,h_{\ve,t})$ locally smoothly converges to $A_t=g_t\cdot (\delbar_E,h_t)$ which is the corresponding $T$-YM flow.
\end{corr}
\subsection{Uhlenbeck limit of the $T$-YM flows}
To prove the existence of a Uhlenbeck limit of the $T$-YM flow $A_t$ in Corollary \ref{cor7}, we need to show some universal estimates on the curvatures of $A_t$.
\begin{lemm}\label{lem8}
Let $A_t$ be the $T$-YM flow on $(E,h_0)$ constructed in Corollary \ref{cor7}.
Then there exists $t_0>0$ and $C=C(t_0)>0$ such that for any $t\ge t_0$, we have
$$
\|F_{A_t}\|_{L^2(X\setminus D,T)}\le C \hbox{ and } \|\Lambda_TF_{A_{t}}\|_{L^{\infty}(X\setminus D)}\le C.
$$
\end{lemm}
\begin{proof}
It follows from  Corollary \ref{cor7} and Lemma \ref{lem4} (2), and the locally smooth convergence of $\omega_{\ve}$ to a K\"{a}hler metric $T|_{X\setminus D}$ on $X\setminus D$.
\end{proof}
\begin{lemm}[see \cite{Sib15} Proposition 6.6]\label{lem9}
Let $A_t$ be the $T$-YM flow on $(E,h_0)$ in Corollary \ref{cor7}.
Then there is a sequence $t_j\to \infty$ such that
$$
\lim_{j\to\infty}\|\nabla_{A_{t_j}}\Lambda_TF_{A_{t_j}}\|_{L^2(X\setminus D,T)}=0.
$$
\end{lemm}
\begin{proof}
The proof is the same as in  \cite[Proposition 6.6]{Sib15}. Since we know the smooth convergence of $A_{\ve,t}$, the argument will be simpler than \cite[Proposition 6.6]{Sib15} where the convergence of $A_{\ve,t}$ is $L^p_{1,\loc}$. 
\end{proof}
Then we obtain the following result:
\begin{corr}\label{cor10}
Let $A_t$ be the $T$-YM flow on $(E,h_0)$ in Corollary \ref{cor7}. Fix $p>n$. Then
there exists a sequence $t_j\to \infty$, a closed subset $Z_{\an}\subset X\setminus D$ of Hausdorff codimension at least 4, a smooth hermitian vector bundle $(E_{\infty},h_{\infty})$ on $X\setminus (D\cup Z_{\an})$ with $T$-YM connection $A_{\infty}$ and isometries $u_j:E\to E_{\infty}$ on $X\setminus (D\cup Z_{\an})$ such that $u_j\cdot A_{t_j}\to A_{\infty}$ weakly in $L^p_{1,\loc}(X\setminus (D\cup Z_{\an}))$. Furthermore, the eigenvalue decomposition of $E_{\infty}$ with respect to $\sqrt{-1}\Lambda_TF_{A_{\infty}}$ gives a holomorphic orthogonal decomposition 
$$
(E_{\infty},A_{\infty},h_{\infty})=\bigoplus_{i=1}^l(Q_{\infty,i},A_{\infty,i},h_{\infty,i})
$$
on $X\setminus (D\cup Z_{\an})$. We also have $\Lambda_TF_{A_{t_j}}\to \Lambda_TF_{A_{\infty}}$ strongly in $L^q(X\setminus (D\cup Z_{\an}))$ for all $1\le q<\infty$. Hereafter, we omit writing the gauge transformations $u_j$ above.
\end{corr}
If we denote the eigenvalues of $\sqrt{-1}\Lambda_TF_{A_{\infty}}$ by $\mu_{\infty}=(\mu_{\infty,1}\ldots,\mu_{\infty,l})$, then we always mean that $Q_{\infty,i}$ is the eigensubbundle corresponding to the eigenvalue $\mu_{\infty,i}$.
\begin{proof}
By Lemma \ref{lem8}, we can apply Theorem \ref{thm2} and hence a Uhlenbeck limit $A_{\infty}$ of the $T$-YM flow $A_t$ exists. Furthermore by Lemma \ref{lem9}, Corollary \ref{cor3} ensures that the limit $A_{\infty}$ is a $T$-YM connection on $(E_{\infty},h_{\infty})$.
\end{proof}
The following lemma produces a nice approximation of $A_{\infty}$.
\begin{lemm}\label{lem11}
Fix $p>n$.
Let $A_{\ve,t}$ be the $\omega_{\ve}$-YM flow on $(E, h_0)$ in Corollary \ref{cor7} and $A_{\infty}$ be the $T$-YM connection on $(E_{\infty},h_{\infty})$ in Corollary \ref{cor10}. Then, there exist sequences $\wtil{\ve_j}\to 0$ and $t_j\to \infty$ such that, for every sequence $\ve_j\to 0$ with $\ve_j<\wtil{\ve_j}$, we have 
$
A_{\ve_j,t_j}\xrightarrow{j\to\infty}A_{\infty} \hbox{ weakly in $L^p_{1,\loc}(X\setminus (D\cup Z_{\an}))$} 
$
and 
$
\|\Lambda_{\omega_{\ve_j}}F_{A_{\ve_j,t_j}}\|_{L^{\infty}(X)}\le C.
$
\end{lemm}
\begin{proof}
The result follows from the smooth convergence of $A_{\ve,t}\xrightarrow{\ve\to0}A_t$ in  Corollary \ref{cor7} and the weak $L^p_{1,\loc}$-convergence $A_{t_j}\to A_{\infty}$ in Corollary \ref{cor10}.
\end{proof}

\section{Convergence of the Hermitian-Yang-Mills type functionals}\label{HYM sec}
Let us introduce the Hermitian-Yang-Mills type functionals following \cite{DW04} to analyze the Harder-Narasimhan type in terms of curvatures.
Let $\beta\ge 1$ be a real number. Then, for a hermitian matrix $v$ with eigenvalues $\lambda_1,\ldots,\lambda_r$, we define
$$
\Phi_{\beta}(v):=\sum_{i=1}^r|\lambda_i|^{\beta}.
$$
We can see that $\Phi_{\beta}$ is well-defined for hermitian endomorphisms of hermitian vector bundles. Thus, we can define the Hermitian-Yang-Mills type functional $\HYM_{\beta,N}^{\omega}(A)$ for $\beta\ge 1$ and $N\ge 0$ and for any integrable unitary connections of a hermitian vector bundle $(E,h_0)$ over a compact K\"{a}hler manifold $(X,\omega)$ by
$$
\HYM_{\beta,N}^{\omega}(A):=\int_X\Phi_{\beta}(\sqrt{-1}\Lambda_{\omega}F_A+N\id_E)\omega^n.
$$
For $\mu=(\mu_1,\ldots,\mu_r)$, we define
$$
\HYM_{\beta,N}(\mu):=\sum_{i=1}^r|\mu_i+N|^{\beta}.
$$
Let $\mu=(\mu_1,\ldots,\mu_r)\in\mathbb{R}^r$ and $\lambda=(\lambda_1,\ldots,\lambda_r)\in\mathbb{R}^r$ be vectors such that $\mu_1\ge\ldots\ge\mu_r$ and $\lambda_1\ge\ldots\ge\lambda_r$. We define $\mu\le\lambda$ if 
$
\sum_{j\le k}\mu_j\le \sum_{j\le k}\lambda_j
$
for all $k=1,\ldots,r$.
Then this functional has the following basic properties:
\begin{lemm}[\cite{DW04}]\label{lem12}
\begin{enumerate}
\item $v\mapsto \left(\int_X\Phi_{\beta}(v)\omega^n\right)^{1/\beta}$ is equivalent to the $L^{\beta}$-norm of the space of hermitian matrices.
\item Let $\mu=(\mu_1,\ldots,\mu_r)$ and $\lambda=(\lambda_1,\ldots,\lambda_r)$ be vectors in $\mathbb{R}^r$.
 If $\mu_r\ge 0$ and $\lambda_r\ge 0$ and if $\HYM_{\beta}(\mu)=\HYM_{\beta}(\lambda)$ for all $\beta$ in a set $A\subset [1,\infty)$ possessing a limit point, then $\mu=\lambda$.
\end{enumerate}
\end{lemm}
The following monotonicity is important for us.
\begin{lemm}[\cite{DW04} Proposition 2.25 and Proposition 2.26]\label{lem12.1}
Let $A_t$ be an $\omega$-YM flow on $(E,h_0)$ where $\omega$ is a smooth K\"{a}hler metric on a compact K\"{a}hler manifold $X$. Then $t\mapsto\HYM_{\beta,N}^{\omega}(A_t)$ monotonically decreases. Furthermore, if $t_j\to\infty$ is a sequence such that $A_{t_j}$ converges to its Uhlenbeck limit $A_{\infty}$ as in Theorem \ref{thm2}, we have $\lim_{j\to\infty}\HYM_{\beta,N}^{\omega}(A_{t_j})=\HYM_{\beta,N}(\mu)$, where $\mu=(\mu_1,\ldots,\mu_l)$ is the $\{\omega\}^{n-1}$-HN type of $(E,\delbar_E)$.
\end{lemm}
Let us fix a holomorphic hermitian vector bundle $(E,\delbar_E,h_0)$ on a compact K\"{a}hler manifold $X$ and a nef and big class $\alpha$ on $X$ such that $D=E_{nK}(\alpha)$ is an snc divisor.
Let $T\in\alpha$ be an adapted current and $\omega_{\ve}\in\alpha+\ve\omega_0$ be its adapted approximation. 
First, we prove the following convergence:
\begin{lemm}\label{lem13}
Let $A_{\ve,t}$ be an $\omega_{\ve}$-YM flow with initial condition $A_0=(\delbar_E,h_0)$ on $(E,h_0)$ and $A_t$ be the $T$-YM flow on $(E,h_0)$ constructed in Corollary \ref{cor7}.
Then, we have
$$
\HYM^{\omega_{\ve}}_{\beta,N}(A_{\ve,t}) \xrightarrow{\ve\to0} \HYM^{T}_{\beta,N}(A_t)
$$
 for every $\beta\ge 1$ and $N\ge 0$ and $t\ge t_0$ where $t_0>0$ is in Lemma \ref{lem4} (2). 
\end{lemm}
\begin{proof}
By Corollary \ref{cor7}, we know $A_{\ve,t}\xrightarrow{\ve\to 0} A_t$ in $C^{\infty}_{\loc}(X\setminus D)$. Hence $\Lambda_{\omega_{\ve}}F_{A_{\ve,t}} \xrightarrow{\ve\to 0} \Lambda_TF_{A_t}$ strongly in $L^p_{\loc}(X\setminus D, T)$ for any $1\le p <\infty$. By Lemma \ref{lem12} (1), it suffices to prove 
\begin{equation}\label{eq12}
\int_{X}|\Lambda_{\omega_{\ve}}F_{A_{\ve,t}}|^p\omega_{\ve}^n\xrightarrow{\ve\to 0} 
\int_{X\setminus D}|\Lambda_TF_{A_t}|^pT^n.
\end{equation}
It follows from the smooth convergence $A_{\ve,t}\xrightarrow{\ve\to0}A_t$ in Corollary \ref{cor7} and the uniform $L^{\infty}$-boundedness of $\Lambda_{\omega_{\ve}}F_{A_{\ve,t}}$ in Lemma \ref{lem4} and of $\Lambda_TF_{A_t}$ in Lemma \ref{lem8}.
\end{proof}
By Corollary \ref{cor10},  we obtain the following.
\begin{lemm}\label{lem14}
Let $A_t$ be the $T$-YM flow on $(E,h_0)$ constructed in Corollary \ref{cor7}. Let $t_j\to\infty$ be a sequence such that $A_{t_j}$ converges to a $T$-YM connection $A_{\infty}$ as in Corollary \ref{cor10}. Let $\mu_{\infty}=(\mu_{\infty,1},\ldots,\mu_{\infty,l})$ be the eigenvalues of $\sqrt{-1}\Lambda_TF_{A_{\infty}}$.
Then, for any $\beta\ge 1$ and $N\ge 0$, we have
$$
\HYM^{T}_{\beta,N}(A_{t_j})\xrightarrow{j\to\infty} \HYM^T_{\beta,N}(A_{\infty})= \HYM_{\beta,N}(\mu_{\infty}).
$$
\end{lemm}
As a direct consequence of Lemma \ref{lem13} and Lemma \ref{lem14}, we obtain the following.
\begin{corr}\label{cor15}
Fix $p>n$.
Let $A_{\ve,t}$ be an $\omega_{\ve}$-YM flow with initial condition $A_0=(\delbar_E,h_0)$ in Corollary \ref{cor7}.
Let $\wtil{\ve}_j\to0$ and $t_j\to\infty$ be sequences in Lemma \ref{lem11}. Let $A_{\infty}$ be the $T$-YM connection on a hermitian vector bundle $(E_{\infty},h_{\infty})$ constructed in Corollary \ref{cor10}. Let $\mu_{\infty}=(\mu_{\infty,1},\ldots,\mu_{\infty,l})$ be the eigenvalues of $\sqrt{-1}\Lambda_TF_{A_{\infty}}$.
Then, there exists $\wtil{\ve_j}\rq{}\to 0$ with $\wtil{\ve_j}>\wtil{\ve_j}\rq{}$ such that, for any sequence $\wtil{\ve_j}\rq{}>\ve_j\to 0$, we have $A_{\ve_j,t_j}\xrightarrow{j\to\infty}A_{\infty}$ weakly in $L^p_{1,\loc}(X\setminus (D\cup Z_{\an}))$, $\|\Lambda_{\omega_{\ve_j}}F_{A_{\ve_j,t_j}}\|_{L^{\infty}(X)}\le C$ and
$$
\HYM^{\omega_{\ve_j}}_{\beta,N}(A_{\ve_j,t_j})\xrightarrow{j\to\infty} \HYM^T_{\beta,N}(A_{\infty})=\HYM_{\beta,N}(\mu_{\infty}).
$$
for any $\beta\ge 1$ and $N\ge 0$. Hereafter, we replace $\wtil{\ve_j}$ by $\wtil{\ve_j}\rq{}$. 
\end{corr}
Then we obtain the following key proposition:
\begin{prop}\label{prop16}
Let $\mu_0$ be the $\alpha^{n-1}$-HN type of $(E,\delbar_E)$ and let $\mu_{\infty}$ be the eigenvalues of the $T$-YM connection $A_{\infty}$ constructed in Corollary \ref{cor10}.
Then there exists a set $A\subset[1,\infty)$ possessing a limit point such that, for every $\beta\in A$ and every $N\ge 0$, we have
$$
\HYM_{\beta,N}(\mu_{\infty})=\HYM_{\beta,N}(\mu_0).
$$
In particular, we have $\mu_{\infty}=\mu_0$.
\end{prop}
\begin{proof}
Fix $\delta>0$ arbitrary.
Let us fix a sequence $\wtil{\ve_k}\to 0$ and $t_j\to\infty$ such that for any $0<\ve_k<\wtil{\ve_k}$,
\begin{equation}\label{eq12.1}
\HYM_{\beta,N}^{\omega_{{\ve_k}}}(A_{{\ve_k},t_k})\to \HYM_{\beta,N}(\mu_{\infty})
\end{equation}
as in Corollary \ref{cor15}. By \cite[Proposition 5.7]{Sib15}, by \cite[Lemma 4.3]{DW04} and by Lemma \ref{lem12.1}, there exists a set $A\subset [1,\infty)$ possessing a limit point such that, for each ${\ve_k}$, there exists $j_k>k$ such that
$$
\HYM_{\beta,N}(\mu_{{\ve_k}})\le \HYM_{\beta,N}^{\omega_{{\ve_k}}}(A_{{\ve_k}, t_{j_k}})<\HYM_{\beta,N}(\mu_{{\ve_k}})+\delta
$$
for every $\beta\in A$.
Since $\mu_{{\ve_k}}\xrightarrow{k\to\infty} \mu_0$ by Proposition \ref{prop3.1}, we have, in $k\to\infty$,
\begin{equation}\label{eq13}
\HYM_{\beta,N}(\mu_0)\le \lim_{k\to\infty}\HYM_{\beta,N}^{\omega_{{\ve_k}}}(A_{{\ve_k},t_{j_k}})\le \HYM_{\beta,N}(\mu_0)+\delta.
\end{equation}
Since $\delta>0$ is chosen arbitrary, it suffices to prove 
\begin{equation}\label{eq14}
\lim_{k\to\infty}\HYM_{\beta,N}^{\omega_{{\ve_k}}}(A_{{\ve_k},t_{j_k}})
=\HYM_{\beta,N}(\mu_{\infty}).
\end{equation}
We can assume that ${\ve_k}\searrow 0$ and $t_j\nearrow\infty$. Since $j_k>k$, we have $t_{j_k}>t_k$. Recall the monotonicity of the HYM type functionals along the $\omega_{\ve}$-YM flows in Lemma \ref{lem12.1}. Together with (\ref{eq12.1}), we have
\begin{align*}
0
&\le \lim_{k\to\infty}\left(\HYM_{\beta,N}^{\omega_{{\ve_k}}}(A_{{\ve_k},t_{j_k}})-\HYM_{\beta,N}(\mu_{{\ve_k}})\right)\notag\\
&\le \lim_{k\to\infty}\left(\HYM_{\beta,N}^{\omega_{{\ve_k}}}(A_{{\ve_k},t_{k}})-\HYM_{\beta,N}(\mu_{{\ve_k}})\right)\notag\\
&=\HYM_{\beta,N}(\mu_{\infty})-\HYM_{\beta,N}(\mu_0).
\end{align*}
Together with (\ref{eq12.1}), we obtain
\begin{align*}
0
&\le \lim_{k\to\infty}\left|\HYM_{\beta,N}^{\omega_{{\ve_k}}}(A_{{\ve_k},t_{j_k}})-\HYM_{\beta,N}^{\omega_{{\ve_{j_k}}}}(A_{{\ve_{j_k}},t_{j_k}}) \right|\notag\\
&\le \lim_{k\to\infty}\left|\left(\HYM_{\beta,N}^{\omega_{{\ve_k}}}(A_{{\ve_k},t_{j_k}})-\HYM_{\beta,N}(\mu_{{\ve_k}})\right)+ \big(\HYM_{\beta,N}(\mu_{{\ve_k}})-\HYM_{\beta,N}(\mu_{\infty})\big) \right|\notag\\
&\hspace{4mm} +\lim_{k\to\infty}\left|\HYM_{\beta,N}(\mu_{\infty})-\HYM_{\beta,N}^{\omega_{{\ve_{j_k}}}}(A_{{\ve_{j_k}},t_{j_k}}) \right|\notag\\
&\le \left|(\HYM_{\beta,N}(\mu_{\infty})-\HYM_{\beta,N}(\mu_0))+(\HYM_{\beta,N}(\mu_{0})-\HYM_{\beta,N}(\mu_{\infty}))\right|+0\notag\\
&=0
\end{align*}
Hence we obtain
\begin{equation}\label{eq15}
\lim_{k\to\infty}\left(\HYM_{\beta,N}^{\omega_{{\ve_k}}}(A_{{\ve_k},t_{j_k}})-\HYM_{\beta,N}^{\omega_{{\ve_{j_k}}}}(A_{{\ve_{j_k}},t_{j_k}})\right)=0.
\end{equation}
Then, by (\ref{eq12.1}) and (\ref{eq15}), we obtain
\begin{align*}
&\lim_{k\to\infty}\HYM_{\beta,N}^{\omega_{{\ve_k}}}(A_{{\ve_k},t_{j_k}})\\
&=\lim_{k\to\infty}\HYM_{\beta,N}^{\omega_{{\ve_{j_k}}}}(A_{{\ve_{j_k}},t_{j_k}})
+\lim_{k\to\infty}\left(\HYM_{\beta,N}^{\omega_{{\ve_k}}}(A_{{\ve_k},t_{j_k}})-\HYM_{\beta,N}^{\omega_{{\ve_k}}}(A_{{\ve_{j_k}},t_{j_k}}) \right)\\
&=\HYM_{\beta,N}(\mu_{\infty}),
\end{align*}
which is exactly (\ref{eq14}). 
\end{proof}

\section{Proof of the main theorem \ref{main thm}}\label{pf of main thm}
Let us fix $(E,\delbar_E)$ a holomorphic vector bundle with a smooth hermitian metric $h_0$ over a compact K\"{a}hler manifold $(X,\omega_0)$ and $\alpha$ a nef and big class such that $D=E_{nK}(\alpha)$ is an snc divisor. 
We also fix $T\in\alpha$ an adapted current and its adapted approximation $\omega_{\ve}\in\alpha+\ve\omega_0$.
Let
\begin{equation}\label{eq15.1}
0=E_0\subset E_1\subset \cdots\subset E_l=E
\end{equation}
be an $\alpha^{n-1}$-HNS filtration and $\Gr^{\HNS}_{\alpha}(E,\delbar_E)=\oplus_{i=1}^lQ_i$ be the associated graded object where $Q_i=E_i/E_{i-1}$.  
Let us consider a $T$-YM flow $A_t$ with initial condition $A_0=(\delbar_E,h_0)$ constructed in Corollary \ref{cor7}. Then its Uhlenbeck limit $A_{\infty}$ defines a graded reflexive sheaf $E_{\infty}=\oplus_{i=1}^{l}Q_{\infty,i}$ on $X\setminus D$ by Corollary \ref{cor10}. To construct an isomorphism between $\Gr^{\HNS}_{\alpha}(E,\delbar_E)$ and $E_{\infty}$, we need to refine the decomposition of $E_{\infty}$. Following \cite{DW04} and \cite{Sib15}, such a refinement will be constructed in Proposition \ref{prop30.0} by letting the above $\alpha^{n-1}$-HNS filtration converge to a filtration of $E_{\infty}$. Using the refined decomposition, we will prove the main theorem \ref{main thm}.

\subsection{Convergence of weakly holomorphic projections}\label{pf of main thm2}
The purpose of this subsection is to prove Corollary \ref{cor18.0} which proves the convergence of $E_{l-1}\subset E_l=E$ in (\ref{eq15.1}). The convergence will be proved by representing subsheaves as weakly holomorphic projections. 
If $\alpha=\omega$ is K\"{a}hler, then $E_{\infty}$ extends to a reflexive sheaf on $X$ and
\cite{DW04} proved that the limit of the $\omega^{n-1}$-HN filtration of $(E,\delbar_E)$ is the $\omega^{n-1}$-HN filtration of $E_{\infty}$. In this paper, $\alpha$ is nef and big, and thus $E_{\infty}$ is defined only on $X\setminus D$ (see Corollary \ref{cor10}).
Then, in Lemma \ref{lem18}, we will prove that the limit of the $\alpha^{n-1}$-HN filtration corresponds to the decomposition of $E_{\infty}$ constructed in Corollary \ref{cor10}.  
The sequence $A_{\ve_j,t_j}$ constructed in Lemma \ref{lem11} plays an important role.
For the proof, we need the following fundamental lemma in linear algebra.
\begin{lemm}[\cite{Fan49}]\label{lem18.0}
Let $V$ be a hermitian vector space of $\dim_{\mathbb{C}}(V)=R<\infty$. Let $L\in \End(V)$ be a hermitian endomorphism of $V$ with eigenvalues $\mu_1\ge\ldots\ge \mu_R$. Let $\pi\in\End(V)$ be an orthogonal projection of $V$ to a subspace of dimension $r$. Then we have 
$$
\Tr(L\circ \pi)\le \mu_1+\cdots+\mu_r.
$$
Assume that $\mu_1=\cdots=\mu_r>\mu_i$ ($i>r$). Then the equality above holds if and only if the image of $\pi$ coincides with the eigen space of $L$ for $\mu_1=\cdots=\mu_r$.
\end{lemm}
We fix a holomorphic hermitian vector bundle $(E,\delbar_E,h_0)$ on a compact K\"{a}hler manifold $X$ and $\alpha$ a nef and big class on $X$ such that $D=E_{nK}(\alpha)$ is an snc divisor. Let $T\in\alpha$ be an adapted current and $\omega_{\ve}\in\alpha+\ve\omega_0$ be its adapted approximation.

Let $A_0$ and $A_1$ be integrable unitary connections on $(E,h_0)$ such that $A_1=g\cdot A_0$ for some $g\in\mathcal{G}^{\mathbb{C}}$. Let $E_i\subset (E,\delbar_E)$ be the i-th factor of the $\alpha^{n-1}$-HNS filtration. Then, a holomorphic map $g_i$ is defined as
\begin{equation}\label{eq15.33}
g_i:(E_i,\delbar_{E_i})\subset (E,\delbar_{A_0})\xrightarrow{g} (E,\delbar_{A_1}).
\end{equation}
\begin{lemm}[cf. \cite{DW04} Lemma 4.5]\label{lem18}
Fix $p>n$.
Let $A_{\ve,t}$ be the $\omega_{\ve}$-YM flow on $(E,h_0)$ in Corollary \ref{cor7} and $\wtil{\ve}_j\to 0$ and $t_j\to\infty$ be sequences Corollary \ref{cor15} and $A_{\infty}$ be the $T$-YM connection on $(E_{\infty}, h_{\infty})$ constructed in Corollary \ref{cor10}.
In particular, $A_{\ve,t}=g_{\ve,t}\cdot (\delbar_E,h_{\ve,t})$ is a sequence of integrable unitary connections on $(E,h_0)$ such that 
\begin{itemize}
\item $A_{\ve_j,t_j}\xrightarrow{j\to \infty} A_{\infty}$ weakly in $L^p_{1,\loc}(X\setminus (D\cup Z_{\an}))$,
\item there exists $C>0$ such that $\|\Lambda_{\omega_{\ve_j}}F_{A_{\ve_j,t_j}}\|_{L^{\infty}(X)}\le C$ and $\|F_{A_{\ve_j,t_j}}\|_{L^2(X,\omega_{\ve_j})}\le C$.
\item The eigenvalues $\mu_{\infty}$ of $\sqrt{-1}\Lambda_TF_{A_{\infty}}$ equal the $\alpha^{n-1}$-HN type $\mu_0$ of $(E,\delbar_E)$. 
\end{itemize}
Let $(E_{\infty},A_{\infty},h_{\infty})=\bigoplus_{i=1}^l(Q_i,A_{\infty,i},h_{\infty,i})$ be the holomorphic orthogonal decomposition constructed in Corollary \ref{cor10}.
Then the following holds.
\begin{enumerate}
\item Let $E_i\subset (E,\delbar_E)$ be the i-th term of the $\alpha^{n-1}$-HN filtration of $(E,\delbar_E)$. Let $\pi^{(i)}_{\ve,t}$ be the weakly holomorphic projection of $(E,\delbar_{A_{\ve,t}}=g_{\ve,t}\cdot \delbar_E,h_0)$ to a holomorphic subbundle $g_{\ve,t,i}(E_i)$, where $g_{\ve,t,i}:(E_i,\delbar_{E_i})\to (E,\delbar_{A_{\ve,t}})$ is defined as in (\ref{eq15.33}). Let $\pi^{(i)}_{\infty}$ be the weakly holomorphic projection of $(E_{\infty},A_{\infty},h_{\infty})$ to $E_{\infty,i}:=\bigoplus_{j\le i}Q_{\infty,j}$.
Then, for every $\wtil{\ve_j}>\ve_j\to 0$, we have
$
\pi^{(i)}_{\ve_j,t_j}\xrightarrow{j\to \infty}\pi^{(i)}_{\infty}
$ 
weakly in $L^p_{2,\loc}(X\setminus (D\cup Z_{\an}))$ and  strongly in $L^p(X\setminus (D\cup Z_{\an}))$.
Furthermore, we have
$
\lim_{j\to\infty}\int_X|\del_{A_{\ve_j,t_j}}\pi^{(i)}_{\ve_j,t_j}|^2\omega_{\ve_j}^n=0$ and $ \del_{A_{\infty}}\pi^{(i)}_{\infty}=0.
$
\item Assume that $(E,\delbar_E)$ is $\alpha^{n-1}$-slope semistable. Let $E_i\subset (E,\delbar_E)$ be the i-th term of an $\alpha^{n-1}$-Seshadri filtration of $(E,\delbar_E)$. Let $\pi^{(i)}_{\ve,t}$ be the weakly holomorphic projection of $(E, \delbar_{A_{\ve,t}}=g_{\ve,t}\cdot \delbar_E, h_0)$ to a holomorphic subbundle $g_{\ve,t,i}(E_i)$. Then, there exists a weakly holomorphic projection $\pi^{(i)}_{\infty}$ of $(E_{\infty},\delbar_{A_{\infty}},h_{\infty})$ such that
$
\pi^{(i)}_{\ve_j,t_j}\xrightarrow{j\to\infty}\pi^{(i)}_{\infty}
$
weakly in $L^p_{2,\loc}(X\setminus (D\cup Z_{\an}))$ and  strongly in $L^p(X\setminus (D\cup Z_{\an}))$. Furthermore, we have $\rk(\pi^{(i)}_{\infty})=\rk(E_i)$ and $\deg_T(\pi^{(i)}_{\infty})=\deg_{\alpha}(E_i)$.
We also have
$\lim_{j\to\infty}\int_X|\del_{A_{\ve_j,t_j}}\pi^{(i)}_{\ve_j,t_j}|^2\omega_{\ve_j}^n=0$ and $\del_{A_{\infty}}\pi^{(i)}_{\infty}=0.$
\end{enumerate}
\end{lemm}
\begin{proof}
(1) Recall that, for each $\delbar_{A_{\ve,t}}=g_{\ve,t}\cdot \delbar_E$, we define $g_{\ve,t,i}$ by
$$
g_{\ve,t,i}:(E_i,\delbar_{E_1})\hookrightarrow (E,\delbar_E)\xrightarrow{g_{\ve,t}} (E,\delbar_{A_{\ve,t}}).
$$
Let $\pi^{(i)}_{\ve,t}$ be the weakly holomorphic projection of $g_{\ve,t}(E_i)\subset (E,\delbar_{A_{\ve,t}},h_0)$. Let us fix a sequence $\ve_j\to 0$ and $t_j\to \infty$ as in Corollary \ref{cor15}. Then, by Corollary \ref{cor15}, we have
\begin{equation}\label{eq25.0}
\int_{X\setminus D}|\Lambda_{\omega_{\ve_j}}F_{A_{\ve_j,t_j}}|^p\omega_{\ve_j}^n\to
\int_{X\setminus D}|\Lambda_{T}F_{A_{\infty}}|^pT^n
\end{equation}
for every $1\le p<\infty$. 
First, we show the convergence of $\pi^{(i)}_{\ve_j,t_j}$.
Recall that the eigenvalues of $\sqrt{-1}\Lambda_TF_{A_{\infty}}$ equal $\mu_0=(\mu_1,\ldots,\mu_l)$ which is the $\alpha^{n-1}$-HN type of $(E,\delbar_E)$ and $\mu_i=\mu_{\alpha}(Q_i)$ by Proposition \ref{prop16}. Then, by Chern-Weil formula and Lemma \ref{lem18.0}, we have
\begin{align}
&\deg_{\omega_{\ve_j}}(E_i)+\int_X\left|\partial_{A_{\ve_j,t_j}}\pi^{(i)}_{\ve_j,t_j}\right|^2\omega_{\ve_j}^n\notag\\
&=\int_X\Tr\left(\sqrt{-1}\Lambda_{\omega_{\ve_j}}F_{A_{\ve_j,t_j}}\circ\pi^{(i)}_{\ve_j,t_j}\right)\omega_{\ve_j}^n\label{eq25}\\
&=\int_X\Tr\left(\sqrt{-1}\Lambda_TF_{A_{\infty}}\circ\pi^{(i)}_{\ve_j,t_j}\right)T^n\notag\\
&\hspace{4mm}+\left(\int_X\Tr\left(\sqrt{-1}\Lambda_{\omega_{\ve_j}}F_{A_{\ve_j,t_j}}\circ\pi^{(i)}_{\ve_j,t_j}\right)\omega_{\ve_j}^n- \int_X\Tr\left(\sqrt{-1}\Lambda_TF_{A_{\infty}}\circ\pi^{(i)}_{\ve_j,t_j}\right)T^n\right)\label{eq26}\\
&\le \sum_{j\le i}(\mu_j\cdot\rk(Q_j))+C(j)\notag\\
&=\deg_{\alpha}(E_i)+C(j)\label{eq27}
\end{align}
where $C(j)\to 0$. Such $C(j)$ exists, since (\ref{eq26}) converges to 0 in $j\to\infty$, which is a consequence of (\ref{eq25.0}) and $|\pi^{(i)}_{\ve_j,t_j}|=1$. Then, by (\ref{eq27}), we obtain
\begin{equation}\label{eq28}
\lim_{j\to\infty}\int_X\left|\partial_{A_{\ve_j,t_j}}\pi^{(i)}_{\ve_j,t_j}\right|^2\omega_{\ve_j}^n=0.
\end{equation}
Recall that $A_{\ve_j,t_j}\to A_{\infty}$ strongly in $C^0_{\loc}(X\setminus (D\cup Z_{\an}))$ by the choice of $\ve_j\to0$ and $t_j\to \infty$ (see Corollary \ref{cor15}) and $|\pi^{(i)}_{\ve_j,t_j}|=1$. Then by (\ref{eq28}), there exists a weakly holomorphic projection $\wtil{\pi}^{(i)}_{\infty}$ of $(E_{\infty}, A_{\infty}, h_{\infty})$ such that
\begin{equation}\label{eq29}
\pi^{(i)}_{\ve_j,t_j}\xrightarrow{j\to\infty}\wtil{\pi}^{(i)}_{\infty}
\end{equation}
weakly in $L^2_{1,\loc}(X\setminus (D\cup Z_{\an}))$ and strongly in $L^p$ for any $1\le p<\infty$ and $\del_{A_{\infty}}\wtil{\pi}^{(i)}_{\infty}=0$. Next we show $\wtil{\pi}^{(i)}_{\infty}={\pi}^{(i)}_{\infty}$.
By the $L^p$-convergence in (\ref{eq29}), we have 
\begin{equation}\label{eq29.1}
\rk(E_i)=\lim_{j\to\infty}\int_X\langle \id_E,\pi^{(i)}_{\ve_j,t_j}\rangle\omega_{\ve_j}^n
=\int_X\langle \id_E,\wtil{\pi}^{(i)}_{\infty}\rangle=\rk(\wtil{\pi}^{(i)}_{\infty}).
\end{equation}
By (\ref{eq25.0}), (\ref{eq28}) and the $L^p$-convergence in (\ref{eq29}), we have, in the limit $j\to \infty$ of (\ref{eq25}),
\begin{equation}\label{eq30}
\deg_{\alpha}(E_i)=\int_X\Tr\left(\sqrt{-1}\Lambda_TF_{A_{\infty}}\circ\wtil{\pi}^{(i)}_{\infty}\right)T^n.
\end{equation}
By (\ref{eq29.1}), we can apply Lemma \ref{lem18.0} and we obtain 
\begin{equation}\label{eq31}
\Tr(\sqrt{-1}\Lambda_TF_{A_{\infty}}\circ\wtil{\pi}^{(i)}_{\infty})\le\sum_{j\le i}\mu_j\cdot\rk(Q_j)= \deg_{\alpha}(E_i).
\end{equation}
Then, by (\ref{eq30}) and (\ref{eq31}), we obtain 
$$
\Tr(\sqrt{-1}\Lambda_TF_{A_{\infty}}\circ\wtil{\pi}^{(i)}_{\infty})=\sum_{j\le i}\mu_j\cdot\rk(Q_j)=\deg_{\alpha}(E_i)
$$
almost everywhere. Then, by the second statement in Lemma \ref{lem18.0}, we obtain the image of $\wtil{\pi}^{(i)}_{\infty}$ is equal to the direct sum of the eigen space of $\sqrt{-1}\Lambda_TF_{A_{\infty}}$ corresponds to $\mu_1,\ldots,\mu_i$, that is, $\bigoplus_{j\le i}Q_{\infty,1}=E_{\infty,i}$. Therefore, we obtain $\wtil{\pi}^{(i)}_{\infty}={\pi}^{(i)}_{\infty}$. Then the weak $L^p_{2,\loc}$ convergence of $\pi^{(i)}_{\ve_j,t_j}$ follows from \cite[Lemma 2.13]{DW04}.
(2) In the semistable case, we can prove the weak convergence to some weakly holomorphic projection with same rank and degree in the same way as above.
\end{proof}
Applying the above Lemma \ref{lem18}, we obtain the following:
\begin{corr}\label{cor18.0}
Fix $p>n$.
We use the same notation as in Lemma \ref{lem18}.
Let 
$$
0=E_0\subset E_1\subset\cdots\subset E_{l-1}\subset E_l=E
$$
be an $\alpha^{n-1}$-HNS filtration of $(E,\delbar_E)$ and let $\pi^{(l-1)}_{\ve_j,t_j}$ be the weakly holomorphic projection of $(E,\delbar_{A_{\ve_j,t_j}}=g_{\ve_j,t_j}\cdot \delbar_E,h_0)$ to $g_{\ve_j,t_j}(E_{l-1})$ as in Lemma \ref{lem18}. Then there exists coherent subsheaves $E_{\infty,l-1}\subset E_{\infty}$ and $Q_{\infty,l}\subset E_{\infty}$ such that the following holds:
\begin{enumerate}
\item If we denote by $\pi^{(l-1)}_{\infty}$ the weakly holomorphic projection of $(E_{\infty},\delbar_{A_{\infty}},h_{\infty})$ to $E_{\infty,l-1}$, then we have
\begin{itemize}
\item $\pi^{(l-1)}_{\ve_j,t_j}\xrightarrow{j\to\infty}\pi^{(l-1)}_{\infty}$ weakly in $L^p_{2,\loc}(X\setminus (D\cup Z_{\an}))$ and strongly in $L^q(X\setminus (D\cup Z_{\an}))$ for any $1\le q<\infty$,
\item $\lim_{j\to\infty}\int_X|\del_{A_{\ve_j,t_j}}\pi^{(l-1)}_{\ve_j,t_j}|^2\omega_{\ve_j}^n=0$, in particular $\del_{A_{\infty}}\pi^{(l-1)}_{\infty}=0$, and
\item $\rk(\pi^{(l-1)}_{\infty})=\rk(E_{l-1})$ and $\deg_T(\pi^{(l-1)}_{\infty})=\deg_{\alpha}(E_{l-1})$.
\end{itemize}
\item We have a holomorphic orthogonal decomposition
$$
(E_{\infty},A_{\infty},h_{\infty})=(E_{\infty,l-1},A_{\infty}^{(l-1)},h_{\infty}^{(l-1)})\oplus (Q_{\infty,l},A_{\infty,l},h_{\infty,l}),
$$
where $A_{\infty}^{(l-1)}=\pi^{(l-1)}_{\infty}\cdot A_{\infty}$ is the $T$-YM connection on $E_{\infty,l-1}$ given by the restriction of $A_{\infty}$ and $A_{\infty,l}$ is the $T$-admissible HE connection on $Q_{\infty,l}$ given by the restriction of $A_{\infty}$. Furthermore, the $T$-HE constant of $A_{\infty,l}$ is $\mu_l=\mu_{\alpha}(Q_l)$ where $Q_l=E/E_{l-1}$.
\end{enumerate}
\end{corr}
\noindent We want to decompose $(E_{\infty,l-1},A_{\infty}^{(l-1)},h_{\infty}^{(l-1)})$. To apply the same argument as above, in the next subsection, we will construct an approximating sequence of the $T$-YM connection $A_{\infty}^{(l-1)}$ on $E_{\infty,l-1}$ such that the sequence satisfies the conditions in Lemma \ref{lem18}.
We will use the following results later.
\begin{prop}\label{prop19}
Fix $p>n$.
Let $\pi^{(i)}_{\ve,t}$ be the weakly holomorphic projection of $(E,\delbar_{A_{\ve,t}}=g_{\ve,t}\cdot \delbar_E,h_0)$ to $g_{\ve,t}(E_i)$ where $E_i$ is the i-th factor of the $\alpha^{n-1}$-HNS filtration as in Corollary \ref{cor18.0}. And let $\pi^{(i)}_t$ be the weakly holomorphic projection of $(E,\delbar_{A_t}=g_t\cdot\delbar_E, h_0)$ to $g_t(E_i)$. Here $h_t=h_0g_t^2$ is the $T$-HYM flow constructed in Proposition \ref{prop6}. 
Let $(E_{\infty},h_{\infty})$ be the hermitian vector bundle with $T$-YM connection $A_{\infty}$ constructed in Corollary \ref{cor10}.
Let $t_j\to\infty$ be a sequence in Lemma \ref{lem18}.
Then we obtain the following:
\begin{enumerate}
\item For each $j$, we have $\pi^{(i)}_{\ve,t_j}\xrightarrow{\ve\to 0}\pi^{(i)}_{t_j}$ in $C^{\infty}_{\loc}(X\setminus D)$ and in $L^q(X\setminus D)$ for any $1\le q<\infty$.
\item $\pi^{(i)}_{t_j}\xrightarrow{j\to\infty}\pi^{(i)}_{\infty}$ weakly in $L^p_{2,\loc}(X\setminus (D\cup Z_{\an}))$ and strongly in $L^q(X\setminus (D\cup Z_{\an}))$ for any $1\le q<\infty$. 
\item There exists a smooth isometry $u^{(i)}_{\ve,j}:(E, h_0)\to (E,h_0)$ on $X\setminus D$ such that $u^{(i)}_{\ve,t_j}\xrightarrow{\ve\to 0}\id_E$ in $C^{\infty}_{\loc}(X\setminus D)$ and in $L^q(X\setminus D)$ for any $1\le q<\infty$. Furthermore,
$$
\left(u^{(i)}_{\ve,t_j}\circ\pi^{(i)}_{\ve,t_j}\circ (u^{(i)}_{\ve,t_j})^{-1}\right)(E)=\pi^{(i)}_{t_j}(E)=g_{t_j}(E_i).
$$
\item There exists a smooth isometry $u^{(i)}_{t_j}:(E,h_0)\to (E_{\infty},h_{\infty})$ on $X\setminus (D\cup Z_{\an})$ such that $u^{(i)}_{t_j}\xrightarrow{j\to\infty} \id_E$ weakly in $L^p_{2,\loc}(X\setminus (D\cup Z_{\an}))$ and strongly in $L^q(X\setminus (D\cup Z_{\an}))$ for any $1\le q<\infty$. Furthermore
$$
\left(u^{(i)}_{t_j}\circ \pi^{(i)}_{t_j}\circ(u^{(i)}_{t_j})^{-1} \right)(E_{\infty})=\pi^{(i)}_{\infty}(E_{\infty})=E_{\infty,i}
$$
\end{enumerate}
\end{prop}
\begin{proof}
(1) By Lemma \ref{lem13}, we have
$$
\int_X|F_{A_{\ve,t}}|^p\omega_{\ve}^n\xrightarrow{\ve\to0} \int_{X\setminus D}|F_{A_{t}}|^pT^n.
$$
Then by the same computation as in (\ref{eq25}), we have
\begin{equation}\label{eq31.1}
\int_X\left|\partial_{A_{\ve,t}}\pi^{(i)}_{\ve,t} \right|^2\omega_{\ve}^n\le C.
\end{equation}
Since $A_{\ve,t}\xrightarrow{\ve\to0}A_{t}$ locally smoothly on $X\setminus D$, we obtain a weakly holomorphic projection $\wtil{\pi}^{(i)}_t$ of $(E,\delbar_{A_t},h_0)$ such that $\pi^{(i)}_{\ve,t}\xrightarrow{\ve\to0}\wtil{\pi}^{(i)}_t$ weakly in $L^2_{1,\loc}(X\setminus D)$.   We show $\wtil{\pi}^{(i)}_t={\pi}^{(i)}_t$. 
Let us define $g_{\ve,t,i}$ by 
$$
g_{\ve,t,i}:(E_i,\delbar_{E_i})\subset (E,\delbar_E)\xrightarrow{g_{\ve,t}} (E,\delbar_{A_{\ve,t}}).
$$
Since $\pi^{(i)}_{\ve,t}$ is the projection to $g_{\ve,t,i}(E_i)$, we have 
\begin{equation}\label{eq32}
\pi^{(i)}_{\ve,t}\circ g_{\ve,t,i}=g_{\ve,t,i}.
\end{equation}
 Since $h_{\ve,t}=h_0g_{\ve,t}^2\to h_0g_t^2=h_t$ smoothly by Proposition \ref{prop6}, we have $g_{\ve,t,i}\xrightarrow{\ve\to0}g_{t,i}$ in $C^{\infty}_{\loc}(X\setminus D)$. Here $g_{t,i}$ is defined from $g_t$ by
$$
g_{t,i}:(E_i,\delbar_{E_i})\hookrightarrow (E,\delbar_E)\xrightarrow{g_t} (E,\delbar_{A_t}).
$$ 
Thus, in $\ve\to0$ of (\ref{eq32}), we obtain 
$$
\wtil{\pi}^{(i)}_{t}\circ g_{t,i}=g_{t,i}.
$$
Thus we have $\im(\pi^{(i)}_t)=g_t(E_i)=\im(g_{t,i})\subset \im(\wtil{\pi}^{(i)}_t)$. Since $\rk(\wtil{\pi}^{(i)}_t)=\rk(\pi^{(i)}_{\ve,t})=\rk(E_i)$ as in Lemma \ref{lem18}, we obtain $\im(\pi^{(i)}_t)=\im(\wtil{\pi}^{(i)}_t)$. Hence, we obtain $\wtil{\pi}^{(i)}_t=\pi^{(i)}_t$.The smooth convergence $\pi^{(i)}_{\ve,t_j}\xrightarrow{\ve\to0}\pi^{(i)}_{t_j}$ follows from the smooth convergence $g_{\ve,t,i}\xrightarrow{\ve\to0}g_{t,i}$. \\
(2) Then (2) follows from the weak $L^p_{2,\loc}$ convergence $\pi_{\ve_j,t_j}^{(i)}\xrightarrow{j\to\infty}\pi_{\infty}^{(i)}$ in Lemma \ref{lem18} and the smooth convergence $\pi_{\ve,t}^{(i)}\xrightarrow{\ve\to0}\pi^{(i)}_t$ in (1) above.
 By (1) and (2), the remaining statements (3) and (4) follow directly from \cite[Lemma 5.12]{Das92}.
\end{proof}

\subsection{Proof of main theorem \ref{main thm}}\label{pf of main thm3}
We fix a holomorphic hermitian vector bundle $(E,\delbar_E,h_0)$ on a compact K\"{a}hler manifold $X$ and $\alpha$ a nef and big class on $X$ such that $D=E_{nK}(\alpha)$ is an snc divisor. Fix also an adapted current $T\in\alpha$ and its adapted approximation $\omega_{\ve}\in\alpha+\ve\omega_0$. We also fix an $\alpha^{n-1}$-HNS filtration of $(E,\delbar_E)$:
\begin{equation}\label{eq32.001}
0=E_0\subset E_1\subset \cdots \subset E_l=E.
\end{equation}
Let $A_{\ve,t}$ be an $\omega_{\ve}$-YM flow with initial condition $A_0=(\delbar_E,h_0)$ on $(E,h_0)$ as in Corollary \ref{cor7} such that $\delbar_{A_{\ve,t}}=g_{\ve,t}\cdot\delbar_{A_0}$ for some $g_{\ve,t}\in\mathcal{G}^{\mathbb{C}}$. Let $\pi^{(i)}_{\ve,t}$ be the weakly holomorphic projection of $(E,\delbar_{A_{\ve,t}},h_0)$ to $g_{\ve,t}(E_i)$ as in Lemma \ref{lem18}. Then we define a unitary connection $A_{\ve,t}^{(i)}$ on $(g_{\ve,t}(E_i),h_0^{(i)})$ induced from $A_{\ve,t}$ on $E$ where $h_0^{(i)}$ is the restriction of $h_0$. Equivalently, we have $\nabla_{A_{\ve,t}^{(i)}}=\pi^{(i)}_{\ve,t}\circ\nabla_{A_{\ve,t}}\circ\pi^{(i)}_{\ve,t}$. We also denote by $A_{\ve,t}^{(i)}=\pi^{(i)}_{\ve,t}\cdot A_{\ve,t}$.
We start with the following estimate.
\begin{lemm}\label{lem20}
Let $A_{\ve,t}$ be the $\omega_{\ve}$-YM flow with initial condition $A_0=(\delbar_E,h_0)$ and $A_{\ve,t}^{(i)}:=\pi_{\ve,t}^{(i)}\cdot A_{\ve,t}$ defined as above.
Then, for any $t>0$, we have
$
\|\Lambda_{\omega_{\ve}}F_{A^{(i)}_{\ve,t}}\|_{L^1(X,\omega_{\ve})}\le C.
$
\end{lemm}
\begin{proof}
Using  Lemma \ref{lem4} (1) and (\ref{eq31.1}), we have for any $t>0$
$$
\int_X\left|\Lambda_{\omega_{\ve}}F_{A^{(i)}_{\ve,t}}\right|\omega_{\ve}^n
\le \int_X\left|\Lambda_{\omega_{\ve}}F_{A_{\ve,t}}\right|\omega_{\ve}^n+\int_X\left|\delbar_{A_{\ve,t}}\pi^{(i)}_{\ve,t}\right|^2\omega_{\ve}^n\le C.
$$
\end{proof}
Since the second fundamental form $\delbar_{A_{\ve,t}}\pi^{(i)}_{\ve,t}$ is not uniformly bounded in $L^{\infty}$, the curvature tensor 
$
\sqrt{-1}\Lambda_{\omega_{\ve}}F_{A^{(i)}_{\ve,t}}
$
is not uniformly bounded in $L^{\infty}$. Then, we consider the $\omega_{\ve}$-YM flow $(A_{s,\ve,t}^{(i)})_s$ on $g_{\ve,t}(E_i)$ with initial condition $A_{\ve,t}^{(i)}$. 
\begin{lemm}\label{lem21}
Let $(A_{s,\ve,t}^{(i)})_s$ be the $\omega_{\ve}$-YM flow on $g_{\ve,t}(E_i)\subset (E,\delbar_{A_{\ve,t}})$ with initial condition $A_{\ve,t}^{(i)}$ as above. Then, 
\begin{enumerate}
\item For any $s>0$, $\ve>0$ and $t>0$, we have
$
\|\Lambda_{\omega_{\ve}}F_{A_{s,\ve,t}^{(i)}}\|_{L^1(X,\omega_{\ve})}\le C.
$
\item At each $s>0$, there exists $C_s>0$ such that, for every $t\ge t_0$ and $\ve>0$,
$
\|\Lambda_{\omega_{\ve}}F_{A^{(i)}_{s,\ve,t}}\|_{L^{\infty}(X)}\le C_s 
$
where $t_0>0$ is defined as in Lemma \ref{lem4} (2).
\item Furthermore, for each $K\Subset X\setminus D$, there exists $C_K>0$ such that for any $\ve>0$, $s\ge 0$ and $t\ge 0$, we have
$
\|\Lambda_{\omega_{\ve}}F_{A_{s,\ve,t}}\|_{L^{\infty}(K)}\le C_K.
$
\end{enumerate}
\end{lemm}
\begin{proof}
Since $A_{s,\ve,t}^{(i)}$ is the $\omega_{\ve}$-YM flow with initial condition $A_{\ve,t}^{(i)}$, we have, by Lemma \ref{lem4} (1) and Lemma \ref{lem20}, 
$$
\int_X|\Lambda_{\omega_{\ve}}F_{A_{s,\ve,t}^{(i)}}|\omega_{\ve}^n
\le \int_X|\Lambda_{\omega_{\ve}}F_{A_{\ve,t}^{(i)}}|\omega_{\ve}^n \le C.
$$
By Lemma \ref{lem1} (3), we have
$$
\left|\Lambda_{\omega_{\ve}}F_{A^{(i)}_{s,\ve,t}}\right|(x)\le \int_XH_{\omega_{\ve}}(x,y,s)\left|\Lambda_{\omega_{\ve}}F_{A^{(i)}_{\ve,t}}\right|(y)\omega_{\ve}(y)^n.
$$
By Lemma \ref{sob ineq}, we have 
$$
H_{\omega_{\ve}}(x,y,s)\le \frac{C}{s^q}\exc(-\frac{d_{\omega_{\ve}}(x,y)}{s})\le C_s
$$ 
at each $s>0$.
Then, together with Lemma \ref{lem20}, we obtain  $\|\Lambda_{\ve}F_{A^{(i)}_{s,\ve,t}}\|_{L^{\infty}(X)}\le C_s$.
\end{proof}
We want to find a limit of $A^{(i)}_{s,\ve,t}$ in $\ve\to 0$. However, the vector bundle where $A^{(i)}_{s,\ve,t}$ is defined is $g_{\ve,t}(E_i)$ which deforms as hermitian vector bundles depending on $\ve$. Then we use the gauge transformations in Proposition \ref{prop19} to fix the vector bundle as follows.
\begin{lemm}\label{lem23}
Let $(A_{s,\ve,t}^{(i)})_s$ be the $\omega_{\ve}$-YM flow as in Lemma \ref{lem21}.
Let $u_{\ve,t}^{(i)}:(E,h_0)\to (E,h_0)$ be the isometry in Proposition \ref{prop19} which, in particular, induces an isometry $u_{\ve,t}^{(i)}:(g_{\ve,t}(E_i),h_0^{(i)})\to (g_t(E_i),h_0^{(i)})$ (refer to Proposition \ref{prop19} for the definition of $g_t$). Here two metrics $h_0^{(i)}$ above are both restrictions of $h_0$ to subbundles $g_{\ve,t}(E_i)$ and $g_t(E_i)$ respectively. Then there exists a smooth hermitian positive definite endomorphism $g_{s,t}^{(i)}: (E_i, h_0^{(i)})\to (g_t(E_i),h_0^{(i)})$ on $X\setminus D$ such that 
$$
u_{\ve,t}^{(i)}\cdot A_{s,\ve,t}^{(i)}\xrightarrow{\ve\to 0} g_{s,t}^{(i)}\cdot (\delbar_{E_i},h_0^{(i)}(g_{s,t}^{(i)})^2)=:A_{s,t}^{(i)}
$$
in $C^{\infty}_{\loc}(X\setminus D)$, where $(\delbar_{E_i},h_0^{(i)}(g_{s,t}^{(i)})^2)$ is the Chern connection of $h_0^{(i)}(g_{s,t}^{(i)})^2$ on $(E_i,\delbar_{E_i})$. The above family $(A_{s,t}^{(i)})_s$ defines a $T$-YM flow on $(g_t(E_i),h_0^{(i)})$ with initial condition $A_t^{(i)}$ which satisfies the following.
\begin{enumerate}
\item For every $s>0$ and $t>0$, we have
$
\|\Lambda_TF_{A_{s,t}^{(i)}}\|_{L^1(X,T)}\le C.
$
\item At each $s>0$, and for every $t\ge t_0$, we have
$
\|\Lambda_TF_{A_{s,t}^{(i)}}\|_{L^{\infty}(X\setminus D)}\le C_s
$
where $t_0>0$ is defined as in Lemma \ref{lem4} (2).
\item Moreover, we have
$\Lambda_{\omega_{\ve}}F_{u_{\ve,t}^{(i)}\cdot A_{s,\ve,t}}^{(i)}\xrightarrow{\ve\to0} \Lambda_TF_{A_{s,t}^{(i)}}$  strongly in $L^q(X\setminus D)$ for $1\le q<\infty$.

\end{enumerate}
\end{lemm}
\begin{proof}
Let us denote by $A_{s,\ve,t}^{(i)}$ the $\omega_{\ve}$-YM flow in $s$ on $(g_{\ve,t}(E_i),h_0^{(i)})$ with initial condition $A_{\ve,t}^{(i)}$. By Lemma \ref{lem0}, there exists a holomorphic isometry
\begin{equation}\label{eq32.0}
g_{s,\ve,t}^{(i)}:(g_{\ve,t}(E_i), \delbar_{A_{\ve,t}^{(i)}},h_0^{(i)}(g_{s,\ve,t}^{(i)})^2)\to (g_{\ve,t}(E_i),\delbar_{A_{s,\ve,t}^{(i)}},h_0^{(i)})
\end{equation}
such that $h_{s,\ve,t}^{(i)}:=h_0^{(i)}(g_{s,\ve,t}^{(i)})^2$ is an $\omega_{\ve}$-HYM flow on the holomorphic vector bundle $(g_{\ve,t}(E_i), \delbar_{A_{\ve,t}^{(i)}})$ with initial condition $h_0^{(i)}$.
We have
\begin{equation}\label{eq32.01}
\nabla_{A_{s,\ve,t}^{(i)}}=g_{s,\ve,t}^{(i)}\circ\nabla_{{(\delbar_{A_{\ve,t}^{(i)}},h_{s,\ve,t}^{(i)})}}\circ(g_{s,\ve,t}^{(i)})^{-1},
\end{equation}
thus we have 
$$
|\Lambda_{\omega_{\ve}}F_{A_{s,\ve,t}^{(i)}}|_{h_0^{(i)}}^2
=|\Lambda_{\omega_{\ve}}F_{(\delbar_{A_{\ve,t}^{(i)}},h_{s,\ve,t}^{(i)})}|_{h_{s,\ve,t}^{(i)}}^2, \hspace{4mm}
|F_{A_{s,\ve,t}^{(i)}}|_{h_0^{(i)},\omega_{\ve}}^2
=|F_{(\delbar_{A_{\ve,t}^{(i)}},h_{s,\ve,t}^{(i)})}|_{h_{s,\ve,t}^{(i)},\omega_{\ve}}^2.
$$
Then the curvatures in RHS above satisfy the same estimates as in Lemma \ref{lem21}.  By Proposition \ref{prop6}, we have
\begin{equation}\label{eq32.1}
\|g_{s,\ve,t}^{(i)}\|_{C^{k/l}(W\times [0,S])}\le C_{W,S}
\end{equation}
for any $k>0$ which is the order of the spacial derivatives, $l>0$ which is the derivatives with respect to $s$, and $W\Subset X\setminus D$. Recall that $h_{\ve,t}=h_0g_{\ve,t}^2$ is the $\omega_{\ve}$-HYM flow on $(E,\delbar_{A_0})$ in (\ref{eq4}) and $g_{\ve,t}$ locally smoothly converges to $g_t$ by Proposition \ref{prop6}. Here $h_t=h_0g_t^2$ is the $T$-HYM flow on $(E,\delbar_{A_0})$ in Proposition \ref{prop6}. Hence, if we define $g_{\ve,t,i}:E_i\hookrightarrow E\xrightarrow{g_{\ve,t}} E$, we have
\begin{equation}\label{eq32.2}
g_{\ve,t,i}\xrightarrow{\ve\to 0} g_{t,i} \hspace{4mm} \hbox{ in $C^{\infty}_{\loc}(X\setminus D)$}.
\end{equation}
where $g_{t,i}:E_i\hookrightarrow E\xrightarrow{g_t}E$. By Proposition \ref{prop19}, we know
\begin{equation}\label{eq32.3}
u_{\ve,t}^{(i)}\xrightarrow{\ve\to0}\id_E \hspace{4mm} \hbox{in $C^{\infty}_{\loc}(X\setminus D)$}.
\end{equation}
Let us consider the composition of holomorphic maps
\begin{equation}\label{eq32.31}
(E_i,\delbar_{E_i})\xrightarrow{g_{\ve,t,i}} (g_{\ve,t}(E_i), \delbar_{A_{\ve,t}^{(i)}})\xrightarrow{g_{s,\ve,t}^{(i)}} (g_{\ve,t}(E_i), \delbar_{A_{s,\ve,t}^{(i)}})\xrightarrow{u_{\ve,t}^{(i)}} (g_t(E_i), \delbar_{u_{\ve,t}^{(i)}\cdot A_{s,\ve,t}^{(i)}}).
\end{equation}
All of them are hermitian with respect to the restriction of $h_0$ to each subbundle of $E$.
Then, by (\ref{eq32.1}), (\ref{eq32.2}) and (\ref{eq32.3}), we obtain a smooth hermitian isomorphism $g_{s,t}^{(i)}:(E_i,h_0^{(i)})\to (g_t(E_i),h_0^{(i)})$ such that
\begin{equation}\label{eq32.4}
u_{\ve,t}^{(i)}\circ g_{s,\ve,t}^{(i)}\circ g_{\ve,t,i}\xrightarrow{\ve\to 0} g_{s,t}^{(i)} \hspace{4mm} \hbox{ in $C^{\infty}_{\loc}(X\setminus D)$.}
\end{equation}
Let us define a hermitian metric $\wtil{h}_{s,\ve,t}^{(i)}$ on $E_i$ by the pullback of $h_0^{(i)}$ on $g_t(E_i)$ along (\ref{eq32.31}):
\begin{equation}\label{eq32.41}
\wtil{h}_{s,\ve,t}^{(i)}:=h_0^{(i)}\cdot(u_{\ve,t}^{(i)}\circ g_{s,\ve,t}^{(i)}\circ g_{\ve,t,i})^2.
\end{equation}
Then we have
\begin{align}
\nabla_{u_{\ve,t}^{(i)}\cdot A_{s,\ve,t}^{(i)}}
&=u_{\ve,t}^{(i)}\circ\nabla_{A_{s,\ve,t}^{(i)}}\circ (u_{\ve,t}^{(i)})^{-1}\notag\\
&=u_{\ve,t}^{(i)}\circ g_{s,\ve,t}^{(i)}\circ\nabla_{(\delbar_{A_{\ve,t}^{(i)}},h_{s,\ve,t}^{(i)})}\circ(g_{s,\ve,t}^{(i)})^{-1}\circ (u_{\ve,t}^{(i)})^{-1}\notag\\
&=(u_{\ve,t}^{(i)}\circ g_{s,\ve,t}^{(i)}\circ g_{\ve,t,i})\circ\nabla_{(\delbar_{E_i},\wtil{h}_{s,\ve,t}^{(i)})}\circ(u_{\ve,t}^{(i)}\circ g_{s,\ve,t}^{(i)}\circ g_{\ve,t,i})^{-1}\label{eq32.42}.
\end{align}
Let us define a hermitian metric $\wtil{h}_{s,t}^{(i)}$ on $E_i$ by 
$$
\wtil{h}_{s,t}^{(i)}:=h_0^{(i)}\cdot (g_{s,t}^{(i)})^2.
$$
Then, by (\ref{eq32.4}), (\ref{eq32.41}) and (\ref{eq32.42}), we obtain that
$$
\nabla_{u_{\ve,t}^{(i)}\cdot A_{s,\ve,t}^{(i)}}
\xrightarrow{\ve\to0} g_{s,t}^{(i)}\circ\nabla_{(\delbar_{E_i},\wtil{h}_{s,t}^{(i)})}\circ (g_{s,t}^{(i)})^{-1}=:\nabla_{A_{s,t}^{(i)}} \hspace{4mm} \hbox{ in $C^{\infty}_{\loc}(X\setminus D)$.}
$$
Since $u_{\ve,t}^{(i)}\cdot A_{s,\ve,t}^{(i)}$ is an $\omega_{\ve}$-YM flow on $(g_t(E_i),h_0^{(i)})$, its smooth limit $A_{s,t}^{(i)}$ is a $T$-YM flow on $(g_t(E_i),h_0^{(i)})$.
Since $u_{\ve,t}^{(i)}$ is isometric, the curvatures of $u_{\ve,t}^{(i)}\cdot A_{s, \ve,t}^{(i)}$ satisfy the same uniform estimates as in Lemma \ref{lem21}. Thus, the smooth limit $A_{s,t}^{(i)}$ also satisfies the same uniform estimates which are the second statements. The last statement follows from Lemma \ref{lem13}.
\end{proof}
Then we use the gauge transformations $u_{t_j}^{(i)}$ in Proposition \ref{prop19} (4) to find a Uhlenbeck limit of $(A_{s,t})_t$ in $t\to \infty$ as follows:
\begin{prop}\label{prop24}
Fix $s>0$ and $p>n$.
Let $t_j\to\infty$ be a sequence in Lemma \ref{lem18}. Let $A_{s,t}^{(i)}$ be a $T$-YM flow in $s$ on $(g_t(E_i),h_0^{(i)})$ constructed in Lemma \ref{lem23}.
Let $u_{t_j}^{(i)}:(E,h_0)\to (E_{\infty},h_{\infty})$ be the isometry in Proposition \ref{prop19}. Then $u_{t_j}^{(i)}\cdot A_{s,t_j}^{(i)}$ is a sequence of integrable unitary connections on $(E_{\infty,i},h_{\infty}^{(i)})$  satisfying 
$
\|\Lambda_TF_{u_{t_j}^{(i)}\cdot A_{s,t_j}^{(i)}}\|_{L^1(X\setminus (D\cup Z_{\an}), T)}\le C,
$
$
\|\Lambda_TF_{u_{t_j}^{(i)}\cdot A_{s,t_j}^{(i)}}\|_{L^{\infty}(X\setminus (D\cup Z_{\an}))}\le C_s \hbox{ and } \|F_{u_{t_j}^{(i)}\cdot A_{s,t_j}^{(i)}}\|_{L^2(X\setminus (D\cup Z_{\an}), T)}\le C_s. 
$\\
In particular, for each $s>0$, there exists a closed subset $Z_{\an}^s\subset X\setminus D$ of Hausdorff codimension at least 4, and a smooth hermitian vector bundle $(E_{\infty,i}^s, h_{\infty,i}^s)$ with an integrable unitary connection $A_{s,\infty}^{(i)}$ such that, up to gauge transformations on $X\setminus (D\cup Z_{\an}^s)$, we have
$
u_{t_j}^{(i)}\cdot A_{s,t_j}^{(i)}\xrightarrow{j\to\infty} A_{s,\infty}^{(i)}
$
weakly in $L^p_{1,\loc}(X\setminus (D\cup Z_{\an}^s))$.
\end{prop}
\noindent We will see in Proposition \ref{prop27} that $A_{s,\infty}^{(i)}$ is $T$-YM.
\begin{proof}
Since $u_{t_j}^{(i)}$ is isometry, the norms of curvatures of $A_{\ve,t_j}$ are the same as those of $u_{t_j}^{(i)}\cdot A_{s,t_j}^{(i)}$. Thus, by Lemma \ref{lem23}, the first statement holds. In particular, by Theorem \ref{thm2}, a Uhlenbeck limit $A_{s,\infty}^{(i)}$ of $u_{t_j}^{(i)}\cdot A_{s,t_j}^{(i)}$ exists.
\end{proof}
For notation, we will omit the gauge transformations $u_{\ve,t}^{(i)}$ in Lemma \ref{lem23} and $u_{t_j}^{(i)}$ in Proposition \ref{prop24} and assume that $A_{s,\ve,t}^{(i)}$ and $A_{s,t}^{(i)}$ are defined on $(E_i,h_0^{(i)})$.
That is, the holomorphic structures $\delbar_{A_{\ve,t}}$ on $E$ preserves $E_i$.
 Furthermore, we can assume that $A_{s,\ve,t}^{(i)}\xrightarrow{\ve\to0} A_{s,t}^{(i)}$ in $C^{\infty}_{\loc}(X\setminus D)$ by Lemma \ref{lem23} and $A_{s,t_j}^{(i)}\xrightarrow{j\to\infty}A_{s,\infty}^{(i)}$ weakly in $L^p_{1,\loc}$ by Proposition \ref{prop24}.  
By the results obtained so far, the proofs of the following lemmas proceed in the same way as in \cite[section 7]{Sib15}.
\begin{lemm}[see \cite{Sib15} Lemma 7.5]\label{lem25}
Let $E_i$ be the i-th term of the $\alpha^{n-1}$-HNS filtration of $(E,\delbar_E)$ in (\ref{eq32.001}).
We denote by $\mu^{(i)}$ the $\alpha^{n-1}$-HN type of $E_i$.
Let $\wtil{\ve_j}\to0$ and $t_j\to\infty$ be sequences in Corollary \ref{cor15} and fix a sequence $0<\ve_j<\wtil{\ve_j}$. 
Let $A_{s,\ve_j,t_j}^{(i)}$ be the $\omega_{\ve}$-YM flow on $E_i$ in Lemma \ref{lem21}.
For each $s>0$, we have
$$
\lim_{j\to\infty}\HYM^{\omega_{\ve_j}}(A_{s,\ve_j,t_j}^{(i)})=\HYM(\mu^{(i)}).
$$
\end{lemm}
Using Lemma \ref{lem25}, we can prove the following.
\begin{lemm}[see \cite{Sib15} Lemma 7.6]\label{lem26}
We use the same notation as Lemma \ref{lem25}.
Then, for every $\wtil{\ve_j}>\ve_j\to0$ and every $s>s_1>0$, we have
$$
\lim_{j\to \infty}\|A_{s,\ve_j,t_j}^{(i)}-A_{s_1,\ve_j,t_j}^{(i)}\|_{L^2(X)}=0.
$$ 
\end{lemm}
As a consequence, we obtain that the Uhlenbeck limit $A_{s,\infty}^{(i)}$ is independent of $s$.
\begin{prop}[see \cite{Sib15} Lemma 7.7]\label{prop27}
Fix $p>n$.
There is a $T$-YM connection $A_{*,\infty}^{(i)}$ on a smooth hermitian vector bundle $E_{\infty,i}^{*}\to X\setminus (D\cup Z_{\an,i}^*)$ where $Z_{\an,i}^*\subset X\setminus D$ is a closed subset of Hausdorff codimension at least 4 with the following property: 
Let $A_{s,t}^{(i)}$ be a sequence of integrable unitary connections on $E_i$ constructed in Lemma \ref{lem23} (also refer to the remark below Proposition \ref{prop24}). Then,
for almost every $s>0$, there is a sequence $t_j\to \infty$ and a closed subset $Z_{\an,i}^s\subset X\setminus D$ such that
$
A_{s,t_j}^{(i)}\xrightarrow{j\to\infty} A_{*,\infty}^{(i)}
$
weakly in $L^p_{1,\loc}(X\setminus (D\cup Z_{\an,i}^s\cup Z_{\an,i}^*))$. Furthermore, we have
$\Lambda_TF_{A_{s,t_j}^{(i)}}\xrightarrow{j\to\infty} \Lambda_TF_{A_{*,\infty}^{(i)}}$ strongly in $L^q(X\setminus (D\cup Z_{\an,i}^s\cup Z_{\an,i}^*))$ for any $1\le q<\infty$.
\end{prop}
Then we obtain the following:
\begin{prop}[see \cite{Sib15} section 7]\label{prop29}
Let $A_{\infty}^{(l-1)}$ be the $T$-YM connection on $E_{\infty,l-1}$ in Corollary \ref{cor18.0}.
Let $A_{*,\infty}^{(l-1)}$ be the $T$-YM connection in Proposition \ref{prop27}. Then we have
$$
A_{*,\infty}^{(l-1)}=A_{\infty}^{(l-1)}.
$$
\end{prop}
As a consequence, we obtain an approximation of $A_{\infty}^{(l-1)}$ admitting curvature estimates:
\begin{corr}\label{cor30}
Let $A_{\infty}^{(l-1)}$ be the $T$-YM connection on $E_{\infty,l-1}$ in Corollary \ref{cor18.0}. For each $\ve>0$ and $t\ge 0$, we denote by $(A_{s,\ve,t}^{(l-1)})_s$ the $\omega_{\ve}$-YM flow on $(E_{l-1},h_0^{(l-1)})$ as in Lemma \ref{lem20} (refer to the remark below Proposition \ref{prop24}).  Fix $s>0$, $p>n$ and sequences $\wtil{\ve_j}\to0$ and $t_j\to\infty$ as in Corollary \ref{cor15}. Then, by replacing $\wtil{\ve_j}$ by a smaller one,  we have for every $\wtil{\ve_j}>\ve_j$,
$
A_{s,\ve_j,t_j}^{(l-1)}\xrightarrow{j\to\infty} A_{\infty}^{(l-1)}
$
weakly in $L^p_{1,\loc}(X\setminus (D\cup Z_{\an}))$. Furthermore, there exists $C_s>0$ at each $s>0$ such that for every $j$,
$
\|\Lambda_{\omega_{\ve_j}}F_{A_{s,\ve_j,t_j}^{(l-1)}}\|_{L^{\infty}(X)}\le C_s.
$
We also have
$\Lambda_{\omega_{\ve_j}}F_{A_{s,\ve_j,t_j}^{(l-1)}}\xrightarrow{j\to\infty} \Lambda_TF_{A_{\infty}^{(l-1)}}$ 
 strongly in $L^q(X\setminus (D\cup Z_{\an}))$ for any $1\le q<\infty$.
\end{corr}
\begin{proof}
Recall that $A_{s,\ve,t_j}^{(i)}\xrightarrow{\ve\to0}A_{s,t_j}^{(i)}$ locally smoothly on $X\setminus D$ by Lemma \ref{lem23} and $A_{s,t_j}^{(i)}\xrightarrow{j\to\infty}A_{\infty}^{(i)}$ weakly in $L^p_{1,\loc}(X\setminus (D\cup Z_{\an}))$ by Proposition \ref{prop27} and Proposition \ref{prop29}. Hence we can find $\wtil{\ve_j}\to 0$ such that for every $\wtil{\ve_j}>\ve_j$, $A_{s,\ve_j,t_j}^{(i)}\xrightarrow{j\to\infty} A_{\infty}^{(i)}$ weakly in $L^p_{1,\loc}$ as in Lemma \ref{lem11}. Then the second statement follows from Lemma \ref{lem21}. The last statement follows from Lemma \ref{lem23} and Proposition \ref{prop27}.
\end{proof}
Then we can prove the following convergence of $\alpha^{n-1}$-HNS filtrations:
\begin{prop}\label{prop30.0}
Fix $p>n$.
Let $(E,\delbar_E,h_0)$ be a holomorphic hermitian vector bundle over a compact K\"{a}hler manifold $X$. Let $\alpha$ be a nef and big class on $X$ such that $D=E_{nK}(\alpha)$ is an snc divisor. Fix an adapted current $T\in\alpha$ and its adapted approximation $\omega_{\ve}\in\alpha+\ve\omega_0$.
Let $0=E_0\subset E_1\subset E_2\subset\cdots\subset E_{l-1}\subset E_l=E$ be an $\alpha^{n-1}$-HNS filtration of $(E,\delbar_{E})$. Without loss of generality, we can assume that each $E_i\subset E_{i+1}$ is a holomorphic subbundle (see Proposition \ref{HNS resol prop}). Let us denote by $h_0^{(i)}$ the restriction of $h_0$ to $E_i$.
Let $(E_{\infty},h_{\infty})$ be a hermitian vector bundle with a $T$-YM connection $A_{\infty}$ constructed in Corollary \ref{cor10}.
Then there exists
\begin{itemize}
\item a filtration $0=E_{\infty,0}\subset E_{\infty,1}\subset \cdots\subset E_{\infty,l}=E_{\infty}$ of $(E_{\infty},\delbar_{A_{\infty}})$ and
\item a sequence of integrable unitary connections $A_j^{(i)}$ on $(E_i,h_0^{(i)})$ for each $i=1,\ldots,l$,
\item and a sequence $\ve_j\to0$
\end{itemize} 
such that the following property holds: 
\begin{enumerate}
\item Without loss of generality, we can assume that the holomorphic structures $\delbar_{A_j^{(i)}}$ on $E_i$ preserves $E_{i-1}$.
\item If we define $Q_{\infty,i}:=E_{\infty,i}/E_{\infty,i-1}$, then we have a holomorphic orthogonal decomposition
$$
(E_{\infty},A_{\infty},h_{\infty})=\bigoplus_{i=1}^l(Q_{\infty,i},A_{\infty,i},h_{\infty,i}),
$$
where $A_{\infty,i}$ is a $T$-admissible HE connection on $Q_{\infty,i}$ with $T$-HE constant $\mu_i=\mu_{\alpha}(Q_i)$, here $Q_i=E_i/E_{i-1}$. We define 
$$
(E_{\infty,i}, A_{\infty}^{(i)},h_{\infty}^{(i)}):=\bigoplus_{j\le i}(Q_{\infty,j},A_{\infty,j},h_{\infty,j})
$$
and $\pi^{(i-1)}_{\infty}$ the weakly holomorphic projection of $(E_{\infty,i},\delbar_{A_{\infty}^{(i)}},h_{\infty}^{(i)})$ corresponding to the holomorphic subbundle $(E_{\infty,i-1},\delbar_{A_{\infty}^{(i-1)}})$.
\item We have $\rk(\pi^{(i-1)}_{\infty})=\rk(E_{i-1})$ and $\deg_T(\pi^{(i-1)}_{\infty})=\deg_{\alpha}(E_{i-1})$.
\item If we denote by $\pi^{(i-1)}_j$ the weakly holomorphic projection of $(E_i,\delbar_{A_j^{(i)}},h_0)$ to a holomorphic subbundle $E_{i-1}$, then $\pi^{(i-1)}_j\xrightarrow{j\to \infty} \pi^{(i-1)}_{\infty}$ weakly in $L^p_{2,\loc}(X\setminus (D\cup Z_{\an}))$ and strongly in $L^q(X\setminus (D\cup Z_{\an}))$ for any $1\le q<\infty$. Furthermore, we have
$
\lim_{j\to\infty}\int_X|\del_{A_j^{(i)}}\pi^{(i-1)}_j|^2\omega_{\ve_j}^n=0.
$
\item We have $A_j^{(i)}\xrightarrow{j\to\infty} A_{\infty}^{(i)}$ weakly in $L^p_{1,\loc}$ and 
$\|\Lambda_{\omega_{\ve_j}}F_{A_j^{(i)}}\|_{L^{\infty}(X)}\le C$ and $\Lambda_{\omega_{\ve_j}}F_{A_j^{(i)}}\xrightarrow{j\to\infty}\Lambda_TF_{A_{\infty}^{(i)}}$ in $L^q(X\setminus (D\cup Z_{\an}))$ for any $1\le q<\infty$.
Furthermore, we have $\del_{A_{\infty}^{(i)}}\pi^{(i-1)}_{\infty}=0.$
\end{enumerate}
\end{prop}
\begin{proof}
We construct such objects inductively. If $i=l$, then we define $A_j^{(l)}:=A_{\ve_j,t_j}$. Then, as remarked below Proposition \ref{prop24}, we can assume that $E_{l-1}$ is preserved by $\delbar_{A_{\ve_j,t_j}}$. By Corollary \ref{cor18.0}, we already know that the weakly holomorphic projections $\pi^{(l-1)}_j:=\pi^{(l-1)}_{\ve_j,t_j}$ to $E_{l-1}\subset (E,\delbar_{A_{\ve_j,t_j}},h_0)$ satisfies (4).
More precisely, $\pi^{(l-1)}_j$ converges to a weakly holomorphic projection $\pi^{(l-1)}_{\infty}$ of $(E_{\infty},A_{\infty},h_{\infty})$ to $E_{\infty,l-1}$ in Corollary \ref{cor18.0} satisfying the condition (3). In particular we have
$
\del_{A_{\infty}}\pi^{(l-1)}_{\infty}=0.
$ 
By Corollary \ref{cor18.0}, we have a holomorphic orthogonal splitting
$$
(E_{\infty},A_{\infty},h_{\infty})=(E_{\infty,l-1},A_{\infty}^{(l-1)},h_{\infty}^{(l-1)})\oplus (Q_{\infty,l},A_{\infty,l},h_{\infty,l})
$$
where $Q_{\infty,l}=E_{\infty}/E_{\infty,l-1}$.
Fix $s=s_{l-1}$. Then,
by Corollary \ref{cor30}, the connections $A_j^{(l-1)}:=A_{s_{l-1},\ve_j,t_j}^{(l-1)}$ on $E_{l-1}$ considered in Corollary \ref{cor30} is the approximating sequence of $A_{\infty}^{(l-1)}$ satisfying the condition (5).
We have constructed the objects satisfying the conditions (1) to (5) when $i=l$. Next we see the construction when $i=l-1$. 

Remark that $0=E_0\subset E_1\subset\cdots\subset E_{l-1}$ is an $\alpha^{n-1}$-HNS filtration of $E_{l-1}$. 
Furthermore $A_{s_{l-1},\ve_j,t_j}^{(l-1)}\xrightarrow{j\to\infty} A_{\infty}^{(l-1)}$ weakly in $L^p_{1,\loc}$ as above, $A_{s_{l-1},\ve,t}^{(l-1)}\xrightarrow{\ve\to0} A_{s_{l-1},t}^{(l-1)}$ in $C^{\infty}_{\loc}(X\setminus D)$ by Lemma \ref{lem23} and $A_{s_{l-1},t_j}^{(l-1)}\xrightarrow{j\to\infty}A_{\infty}^{(l-1)}$ weakly in $L^p_{1,\loc}$ by Proposition \ref{prop24}, Proposition \ref{prop27} and Proposition \ref{prop29}.  
Hence we can apply Proposition \ref{prop19} to $E_{l-2}\subset E_{l-1}$ and thus we can assume that $\delbar_{A_j^{(l-1)}}=\delbar_{A_{s_{l-1},\ve_j,t_j}^{(l-1)}}$ on $E_{l-1}$ preserves the subbundle $E_{l-2}$ as remarked below Proposition \ref{prop24}. In particular, we obtain the following: Let $\pi^{(l-2)}_{j}$ be the weakly holomorphic projection of $E_{l-2}\subset (E_{l-1},\delbar_{A_{j}^{(l-1)}},h_0)$. Then by the estimates on curvatures of $A_j^{(l-1)}=A_{s_{l-1},\ve_j,t_j}^{(l-1)}$ in (5), Corollary \ref{cor18.0} implies that there exists a weakly holomorphic projection $\pi^{(l-2)}_{\infty}$ of $(E_{\infty,l-1},\delbar_{A_{\infty}^{(l-1)}},h_0)$ with same rank and degree with $E_{l-2}$ such that $\pi^{(l-2)}_j\xrightarrow{j\to\infty}\pi^{(l-2)}_{\infty}$ weakly in $L^p_{2,\loc}$ and
$
\lim_{j\to\infty}\|\del_{A_j^{(l-1)}}\pi^{(l-2)}_j\|_{L^2(X,\omega_{\ve_j})}=0.
$
In particular we have $\del_{A_{\infty}^{(l-1)}}\pi_{\infty}^{(l-2)}=0$. Thus, if we denote by $E_{\infty,l-2}:=\im(\pi^{(l-2)}_{\infty})\subset (E_{\infty,l-1},\delbar_{A_{\infty}^{(l-1)}})$, we obtain a holomorphic orthogonal splitting
$$
(E_{\infty,l-1},A_{\infty}^{(l-1)},h_{\infty}^{(l-1)})=(E_{\infty,l-2},A_{\infty}^{(l-2)},h_{\infty}^{(l-2)})\oplus (Q_{\infty,l-1},A_{\infty,l-1},h_{\infty,l-1})
$$
where $Q_{\infty,l-1}=E_{\infty,l-1}/E_{\infty,l-2}$.
Let us denote by $A_{s,j}^{(l-2)}$ the $\omega_{\ve_j}$-YM flow on $E_{l-2}$ with initial condition $\pi^{(l-2)}_j\cdot A_{j}^{(l-1)}$, which is a restriction of $A_j^{(l-1)}$ to $E_{l-2}\subset E_{l-1}$. Fix $s=s_{l-2}>0$. 
 Then, in the same way as in the case when $i=l$ above, we can see that the sequence of connections $A_j^{(l-2)}:=A_{s_{l-2},j}^{(l-2)}$ defines an approximating sequence of $A_{\infty}^{(l-2)}$ satisfying the condition (5).
Then we can prove the result for general $i$ in the same way.
\end{proof}

Then we can prove the main theorem.
\begin{theo}\label{thm31}
Let $(E,\delbar_E,h_0)$ be a holomorphic hermitian vector bundle over a compact K\"{a}hler manifold $X$ with a nef and big class $\alpha$. Let $T\in\alpha$ be an adapted current. Let $A_t$ be the $T$-YM flow on $(E,h_0)$ in Corollary \ref{cor7} and $A_{\infty}$ be the $T$-YM connection on a hermitian vector bundle $(E_{\infty},h_{\infty})$ given by a Uhlenbeck limit of $A_t$ as in Corollary \ref{cor10}.
Let 
$$
0=E_0\subset E_1\subset \cdots \subset E_l=E
$$ 
be an $\alpha^{n-1}$-HNS filtration of $(E,\delbar_E)$ and 
$$
0=E_{\infty,0}\subset E_{\infty,1}\subset \cdots \subset E_{\infty,l}=E_{\infty}
$$
be the corresponding filtration of $(E_{\infty},\delbar_{A_{\infty}})$ constructed in Proposition \ref{prop30.0}. Then there exists a parallel isomorphism of graded vector bundles:
$$
g_{\infty}:\Gr^{\HNS}_{\alpha}(E,\delbar_E)=\bigoplus_{i=1}^l(Q_i,\delbar_{Q_i}, h_i)
\simeq \bigoplus_{i=1}^l(Q_{\infty,i},\delbar_{A_{\infty,i}},h_{\infty,i})=(E_{\infty},\delbar_{A_{\infty}},h_{\infty}),
$$
where $h_i$ is the $T$-adapted HE metric on $(Q_i=E_i/E_{i-1},\delbar_{Q_i})$ and the RHS is the holomorphic orthogonal splitting constructed in Proposition \ref{prop30.0}. More precisely, the restriction of $g_{\infty}$ above defines a holomorphic isomorphism
$$
g_{\infty,i}: (Q_i,\delbar_{Q_i},h_i)\to (Q_{\infty,i},\delbar_{A_{\infty,i}},h_{\infty,i})
$$
such that $\nabla_{(\delbar_{A_{\infty,i}},h_{\infty,i})}\circ g_{\infty,i}=g_{\infty,i}\circ\nabla_{(\delbar_{Q_i},h_i)}$.
\end{theo}
\begin{proof}
Without loss of generality, we can assume that $D=E_{nK}(\alpha)$ is an snc divisor and the $\alpha^{n-1}$-HNS filtration of $(E,\delbar_E)$ is given by holomorphic subbundles. Let $s_D$ be the defining section of $D$.
We fix an adapted current $T\in \alpha$ and its adapted approximation $\omega_{\ve}\in\alpha+\ve\omega_0$.
We use a sequence of integrable unitary connections $A_j^{(i)}$ on $E_i$ constructed in Proposition \ref{prop30.0}. By Proposition \ref{prop30.0}, we can assume that $E_{i-1}$ is a holomorphic subbundle of $(E_i,\delbar_{A_j^{(i)}})$.
We then denote by $g_{j,i}$ the natural inclusion of holomorphic subbundle
$$
g_{j,i}:(E_{i-1},\delbar_{E_{i-1}})\hookrightarrow (E_i,\delbar_{A_j^{(i)}}).
$$
We first consider the case when $i=2$. 
By definition of $\alpha^{n-1}$-HNS filtrations, we know that $(E_1,\delbar_{E_1})$ is $\alpha^{n-1}$-slope stable. Thus we have a $T$-adapted HE metric $h_1$ on $(E_1,\delbar_{E_1})$ and an $\omega_{\ve}$-HE metric $h_{\ve,1}$ on $(E_1,\delbar_{E_1})$ such that $h_{\ve,1}\xrightarrow{\ve\to0} h_1$ in $C^{\infty}_{\loc}(X\setminus D)$.
We consider (possibly non-isometric) holomorphic maps
$$
g_{j,1}:(E_1,\delbar_{E_1}, h_{j,1}:=h_{\ve_j,1})\to (E_2,\delbar_{A_j^{(2)}}, h_0)
$$
defined above. Since $\|\Lambda_{\omega_{\ve_j}}F_{A_j^{(2)}}\|_{L^{\infty}(X)}\le C$ by Proposition \ref{prop30.0} (5), we have
\begin{align}
&\Delta_{\omega_{\ve_j}}|g_{j,1}|_{h_{j,1}^*\otimes h_0}^2\notag\\
&=\left|\del_{(\delbar_{E_1},h_{j,1})\otimes A_j^{(2)}}g_{j,1} \right|^2_{h_{j,1}^*\otimes h_0}
-\langle g_{j,1}, \sqrt{-1}\Lambda_{\omega_{\ve_j}}F_{A_j^{(2)}}\circ g_{j,1}-g_{j,1}\circ \sqrt{-1}\Lambda_{\omega_{\ve_j}}F_{(\delbar_{E_1},h_{j,1})}\rangle\label{eq40}\\
&\ge -C|g_{j,1}|_{h_{j,1}^*\otimes h_0}^2\label{eq41}
\end{align}
Multiplying $|g_{j,1}|_{h_{j,1}^*\otimes h_0}^2$ to both sides of (\ref{eq41}) above and integrating with respect to $\omega_{\ve_j}^n$, we have
$$
\int_X\left|\nabla|g_{j,1}|_{h_{j,1}^*\otimes h_0}^2\right|^2\omega_{\ve_j}^n\le C\int_X|g_{j,1}|_{h_{j,1}^*\otimes h_0}^4\omega_{\ve_j}^n.
$$
Then we normalize $g_{j,1}$ so that 
\begin{equation}\label{eq41.1}
\int_X|g_{j,1}|_{h_{j,1}^*\otimes h_0}^4\omega_{\ve_j}^n=1.
\end{equation}
Then by Lemma \ref{sob ineq} (1) and by Moser iteration, we obtain
\begin{equation}\label{eq42}
\sup_X|g_{j,1}|_{h_{j,1}^*\otimes h_0}\le C
\end{equation}
where $C>0$ depends on the Sobolev constant of $(X,\omega_{\ve_j})$ which is uniformly bounded by Lemma \ref{sob ineq}. Then, by integrating (\ref{eq40}), we obtain
\begin{equation}\label{eq43}
\int_X\left|\del_{(\delbar_{E_1},h_{j,1})\otimes A_j^{(2)}}g_{j,1} \right|^2_{h_{j,1}^*\otimes h_0}\omega_{\ve_j}^n\le C.
\end{equation}
Then $h_{j,1}\xrightarrow{j\to\infty} h_1$ in $C_{\loc}^{\infty}(X\setminus D)$, $A_j^{(2)}\xrightarrow{j\to\infty} A_{\infty}^{(2)}$ weakly in $L^p_{1,\loc}$, thus strongly in $C^0_{\loc}$, Proposition \ref{prop30.0} (5), the uniform estimates (\ref{eq42}) and (\ref{eq43}) imply that
$$
g_{j,1}\xrightarrow{j\to\infty} g_{\infty,1}\hspace{4mm} \hbox{ weakly in $L^2_{1,\loc}$.}
$$
Then $g_{j,1}$ converges to $g_{\infty,1}$ strongly in $L^2_{\loc}$, hence the uniform $L^{\infty}$-estimate (\ref{eq42}) ensures that 
$$
\lim_{j\to\infty}\int_X|g_{j,1}|_{h_{j,1}^*\otimes h_0}^p\omega_{\ve_j}^n
=\int_X|g_{\infty,1}|^p_{h_1^*\otimes h_0}\langle T^n\rangle
$$
for any $1\le p<\infty$. Therefore we have $g_{\infty,1}\ne0$ by (\ref{eq41.1}). Since $\delbar_{(\delbar_{E_1},h_{j,1})^{\vee}\otimes A_j^{(2)}}g_{j,1}=0$, $\delbar_{(\delbar_{E_1},h_{j,1})^{\vee}}=\delbar_{E_1^{\vee}}$ and $A_j^{(2)}\xrightarrow{j\to\infty}A_{\infty}^{(2)}$ in $C^0_{\loc}$ by Proposition \ref{prop30.0} (5), we have that $(\delbar_{E_1^{\vee}}\otimes \delbar_{A_{\infty}^{(2)}})g_{\infty,1}=0$. That is,
\begin{equation}\label{eq44}
g_{\infty,1}:(E_1,\delbar_{E_1},h_1)\to (E_{\infty,2},\delbar_{A_{\infty}^{(2)}}, h_{\infty})
\end{equation}
is a non-zero holomorphic map. We show that above $g_{\infty,1}$ is parallel. Let us consider 
$$
E_1\xrightarrow{g_{j,1}} (E_2,A_j^{(2)}, h_0^{(2)})\xrightarrow{(\pi^{(1)}_j)^{\perp}} E_2.
$$
Here $(\pi^{(1)}_j)^{\perp}=\id_{E_2}-\pi^{(1)}_j$ which is the projection of $E_1^{\perp}\subset (E_2,h_0^{(2)})$. 
By Proposition \ref{prop30.0} (4), we have that $\pi_j^{(1)}$ converges to a weak holomorphic projection $\pi_{\infty}^{(1)}$.
Since $(\pi^{(1)}_j)^{\perp}\circ g_{j,1}=0$, we have, in $j\to\infty$, that $(\pi^{(1)}_{\infty})^{\perp}\circ g_{\infty,1}=0$.
Recall the holomorphic orthogonal splitting
$$
(E_{\infty,2}, A_{\infty}^{(2)},h_{\infty}^{(2)})=(E_{\infty,1},A_{\infty}^{(1)},h_{\infty}^{(1)})\oplus (Q_{\infty,2},A_{\infty,2}, h_{\infty,2})
$$
in Proposition \ref{prop30.0} (2).
Since $\pi^{(1)}_{\infty}$ is the projection to $E_{\infty,1}$ by Proposition \ref{prop30.0}, we have that $(\pi^{(1)}_{\infty})^{\perp}$ is the projection to $Q_{\infty,2}$. Thus we obtain that the image of $g_{\infty,1}$ is contained in the kernel of $(\pi^{(2)}_{\infty})^{\perp}$, that is, in $E_{\infty,1}=Q_{\infty,1}$. 
Thus, $g_{\infty,1}$ defines a non-zero holomorphic map 
\begin{equation}\label{eq45}
g_{\infty,1}:(E_1,\delbar_{E_1},h_1)\to (E_{\infty,1},\delbar_{A_{\infty,1}},h_{\infty,1}),
\end{equation}
where
$A_{\infty,1}$ is a $T$-admissible HE connection on $(Q_{\infty,1},h_{\infty,1})$ with $T$-HE constant $\mu_1=\mu_{\alpha}(E_1)$ (refer to Proposition \ref{prop30.0} (2)). Since $h_1$ is a $T$-adapted HE metric on $(E_1,\delbar_{E_1})$, its $T$-HE constant is also $\mu_1$ (refer to \cite{Jin26}).
Let $h_D$ be a $T$-adapted HE metric on $\mathcal{O}(D)$ (see Theorem \ref{KH corr nef big}) and $s_D$ be a defining section of $D$. Then the $T$-HE constant of $h_D$ equals $\mu_{\alpha}(D)=0$ (refer to \cite{Jin26}). Fix a large integer $k>0$.
Then we have
$$
\Delta_T|s_D^kg_{\infty,1}|_{h_D^k\otimes h_1^*\otimes h_0}^2
=\left|\del_{(\delbar_{\mathcal{O}(kD)}, h_D^k)\otimes(\delbar_{E_1},h_1)\otimes A_{\infty}^{(2)}}(s_D^kg_{\infty,1})\right|^2_{h_D^k\otimes h_1^*\otimes h_0}.
$$
By the argument in \cite[Lemma 3.16]{Jin26}, we can prove that both sides of the above equation are integrable with respect to $T^n$ and 
\begin{equation}\label{eq45.1}
\int_{X\setminus D}\Delta_T|s_D^kg_{\infty,1}|_{h_D^k\otimes h_1^*\otimes h_0}^2T^n=0.
\end{equation}
Therefore we have
$$
\int_{X\setminus D}\left|\del_{(\delbar_{\mathcal{O}(kD)}, h_D^k)\otimes(\delbar_{E_1},h_1)\otimes A_{\infty}^{(2)}}(s_D^kg_{\infty,1})\right|^2_{h_D^k\otimes h_1^*\otimes h_0}T^n=0.
$$
Thus $g_{\infty,1}$ in (\ref{eq45}) is parallel. Since it is non-zero and $(E_1,\delbar_{E_1})$ is $\alpha^{n-1}$-slope stable, we obtain that $g_{\infty,1}$ in (\ref{eq45}) is an isomorphism by Lemma \ref{simple stable}. 

Next we see the construction of parallel isomorphism $Q_2\to Q_{\infty,2}$. 
Let us consider the natural injection
$$
g_{j,2}:(Q_2,\delbar_{Q_2},h_{j,2})^{\vee}\hookrightarrow (E_2,\delbar_{A_j^{(2)}},h_0)^{\vee}
$$
induced by the natural projection $E_2\to Q_2$. Here $h_{j,2}$ is the $\omega_{\ve_j}$-HE metric on $(Q_2,\delbar_{Q_2})$. Let us denote by $h_2$ the $T$-adapted HE metric on $(Q_2,\delbar_{Q_2})$ such that $h_{j,2}\xrightarrow{j\to\infty} h_2$ in $C^{\infty}_{\loc}(X\setminus D)$. In the same way as above, we can construct a non-zero holomorphic map
$$
g_{\infty,2}:(Q_2,\delbar_{Q_2},h_2)^{\vee}\to (E_{\infty,2},\delbar_{A_{\infty}^{(2)}},h_{\infty}^{(2)})^{\vee}
$$
such that $g_{j,2}\xrightarrow{j\to\infty} g_{\infty,2}$ weakly in $L^2_{1,\loc}$ and strongly in $L^p$.
Let us consider the composition 
$$
Q_2^{\vee} \xrightarrow{g_{j,2}} E_2^{\vee}\xrightarrow{g_{j,1}^{\vee}} E_1^{\vee}.
$$
Since $g_{j,2}$ is the dual of the natural quotient $E_2\to Q_2=E_2/E_1$ and $g_{j,1}$ is the natural inclusion $E_1\subset (E_2,\delbar_{A_j^{(2)}})$, we have $g_{j,1}^{\vee}\circ g_{j,2}=0$. Since $g_{j,1}\to g_{\infty,1}$ and $g_{j,2}\to g_{\infty,2}$ in $L^2_{\loc}$, we have
$
g_{\infty,1}^{\vee}\circ g_{\infty,2}=0.
$
Hence the image of $g_{\infty,2}$ is contained in the kernel of $g_{\infty,1}^{\vee}$, which equals $Q_{\infty,2}^{\vee}$.
Hence we obtain that $g_{\infty,2}$ induces a non-zero holomorphic map
\begin{equation}\label{eq46}
g_{\infty,2}:(Q_2,\delbar_{Q_2},h_2)^{\vee}\to (Q_{\infty,2},\delbar_{A_{\infty,2}},h_{\infty,2})^{\vee}
\end{equation}
Here we remark that $A_{\infty,2}$ is the $T$-admissible HE connection on $Q_{\infty,2}$ with Einstein constant $\mu_2=\mu_{\alpha}(Q_2)$ (refer to Proposition \ref{prop30.0} (5)). Thus, by the same argument as above, we know that $g_{\infty,2}$ above is non-zero parallel. Since $(Q_2,\delbar_{Q_2})^{\vee}$ is $\alpha^{n-1}$-slope stable, we obtain that $g_{\infty,2}$ in (\ref{eq46}) is an isomorphism.  Then $g_{\infty,1}$ in (\ref{eq45}) and $g_{\infty,2}$ in (\ref{eq46}) define parallel isomorphisms
\begin{align*}
g_{\infty,1}\oplus (g_{\infty,2}^{\vee})^{-1}
:\Gr^{\HNS}_{\alpha}(E_2,\delbar_{E_2})
&=\bigoplus_{i=1}^2(Q_i,\delbar_{Q_i}, h_i)\\
&\simeq \bigoplus_{i=1}^2(Q_{\infty,i},A_{\infty,i},h_{\infty,i})=(E_{\infty,2},A_{\infty,2},h_{\infty,2}),
\end{align*}
Repeating the same argument, we can prove the theorem.
\end{proof}

\end{document}